\documentclass[11pt]{article}

\usepackage[utf8]{inputenc}
\usepackage{amsmath}
\usepackage{amsfonts,amssymb,amsthm}
\usepackage{fullpage, ifthen, color, graphicx, hyperref, enumerate, verbatim, booktabs}
\usepackage{paralist}
\usepackage{graphicx}
\usepackage{float}

\newtheorem{theorem}{Theorem}
\newtheorem{lemma}[theorem]{Lemma}
\newtheorem{corollary}[theorem]{Corollary}

\newcommand{\ds}{\displaystyle}

\title{
	Maximal independent sets and maximal matchings in series-parallel and related graph classes
}

\author{
Michael Drmota
	\thanks{
        Institute of Discrete Mathematics and Geometry, Technische Universität Wien, Austria.
        Supported by the Special Research Program F50 {\em Algorithmic and Enumerative Combinatorics} of the Austrian Science Fund.
        Email: {\tt michael.drmota@tuwien.ac.at}.
    }
\and
Lander Ramos
	\thanks{
        Departament de Matemàtiques, Universitat Politècnica de Catalunya, Barcelona, Spain.
        Email: {\tt landertxu@gmail.com}.
    }
\and
Clément Requilé
	\thanks{
        Institute of Discrete Mathematics and Geometry, Technische Universität Wien, Austria.
	    Supported by the Special Research Program F50 {\em Algorithmic and Enumerative Combinatorics} of the Austrian Science Fund.
        Part of the work presented in this paper was carried out while the author was affiliated with the  Institute for Algebra, Johannes Kepler Universität Linz, Austria.
        E-mail: {\tt clement.requile@tuwien.ac.at}.
    }
\and
Juanjo Rué
	\thanks{
        Departament de Matemàtiques, Universitat Politècnica de Catalunya and Barcelona Graduate School of Mathematics, Spain.
        Partially supported by the FP7-PEOPLE-2013-CIG project CountGraph (ref. 630749),
	    the Spanish MICINN projects MTM2014-54745-P and MTM2017-82166-P
        and the María de Maetzu research grant MDM-2014-0445.
        Email: {\tt juan.jose.rue@upc.edu}.
    }
}

\begin{document}

\maketitle

\begin{abstract}
	The goal of this paper is to obtain quantitative results on the number and on the size of maximal independent sets and maximal matchings in several block-stable graph classes that satisfy a proper sub-criticality condition.
	In particular we cover trees, cacti graphs and series-parallel graphs.
	The proof methods are based on a generating function approach and a proper singularity analysis of solutions of implicit systems of functional equations in several variables.
	As a byproduct, this method extends previous results of Meir and Moon for trees [Meir, Moon: On maximal independent sets of nodes in trees, Journal of Graph Theory 1988].
\end{abstract}

%
%
\newpage

\section{Introduction}\label{sec:introduction}

In this paper we consider labelled, loopless and simple graphs only.
For a graph $G=(V(G),E(G))$, a subset $J$ of $V(G)$ is said to be \emph{independent} if for any pair of vertices $x$ and $y$ contained in $J$, the edge $\{x,y\}$ does not belong to $E(G)$.
An independent set $J$ of a graph $G$ is said to be {\em maximal} if any other vertex of $G$ that is not contained in $J$ is adjacent to at least one vertex of $J$.
A subset $N$ of the edge set $E(G)$ is called a {\it matching} if every vertex $x$ of $G$ is incident to at most one edge of $N$.
A matching $N$ is called {\em maximal} if it cannot be extended to a bigger matching by adding an edge from $E(G)\setminus N$.

More precisely, let $\mathcal{G}$ denote a class of vertex-labelled graphs (vertices on a graph on $n$ vertices are labelled by $\{1, 2, \ldots, n\}$).
We denote by $\mathcal{G}_n$ the set of graphs in $\mathcal{G}$ with $n$ vertices.
For $G\in \mathcal{G}$ we denote by $I(G)$ the set of maximal independent sets of $G$ and by
\begin{equation*}
    \mathcal{I}_n = \bigcup_{ G \in \mathcal{G}_n}  I(G) \times \{G\}
\end{equation*}
the system of all maximal independent sets of graphs in $\mathcal{G}_n$.
That is, every maximal independent set $J$ is {\it indexed} by the corresponding graph, this is formally done by taking pairs $(J,G)$.
Similarly, we denote by $M(G)$ the set of maximal matchings of $G$ and by
\begin{equation*}
    \mathcal{M}_n = \bigcup_{ G \in \mathcal{G}_n}  M(G) \times \{G\}
\end{equation*}
the system of all maximal matchings of graphs of size $n$.
The purpose of this paper is to enumerate maximal independent sets and maximal matchings by means of symbolic methods.
Later, by using complex analytic tools, we will study their size distribution.

Our research will focus on graph classes defined by an analytic condition that is reminiscent to the subcritical condition in constrained graph classes.
Roughly speaking, a graph class is called \emph{subcritical} if the largest block on a uniformly at random graph with $n$ vertices in the class has $O(\log(n))$ vertices, a block being a 2--connected component that is maximal in the sense of inclusion.
The precise analytic definition (which is done in terms of generating functions) can be found in Subsection \ref{subs:subcr} (see \cite{Bernasconi2009,DrFuKaKrRu11}).
Subcritical graphs have typically a tree--like structure and share several properties with trees.
For instance, forests, cacti graphs, outerplanar graphs, series--parallel graphs and more generally graph families defined by a finite set of 3--connected components are subcritical graph families \cite{GiNoRu13}.
Nowadays, we have a quite good understanding of the shape of a uniformly at random graph of size $n$: we know exact results for subgraph statistics \cite{subgraphs}, their Benjamini--Schramm convergence \cite{schben15}, their maximum degree \cite{DrNo2013} and their scaling limits towards the continuous random tree \cite{panagiotou2015}.
Let us also mention that the study of subcritical graph classes is intimately linked to the understanding of the random planar graph model.
In fact, it it was conjectured that a graph class defined by a set of excluded minors is subcritical if and only if at least one of the excluded graphs is planar (see \cite{NoyICM}).
One direction of this conjecture is false: by the work of \cite{georWa18}, there exists a minor--closed addable subcritical graph family which contains all planar graphs. However, the converse implication in the conjecture is still unknown.

To clarify our setting, let us remark that our research will not be applied to subcritical graph classes in the classical sense of the papers \cite{Bernasconi2009,DrFuKaKrRu11}.
In fact, our results can be applied to graph classes that satisfy some \emph{extended} sub--criticality conditions which are slightly different from the usual subcritical analytic condition.
We will make this definition precise in Subsection~\ref{subs:subcr}.
In fact, we will show that relevant subcritical graph classes (trees, cacti graphs and series--parallel graphs) fit into this new scheme.

These techniques will be used to obtain precise enumerative results on $\mathcal{I}_n$ and $\mathcal{M}_n$.
In particular, we will apply our method to three important graph families: Cayley trees, cacti graphs and series--parallel graphs.
For the mentioned graph classes we have the following universal structure in the asymptotic enumeration formula for the number of graphs on $n$ vertices, for $n$ large enough:
\begin{equation}\label{eqgn}
    g_n = |\mathcal{G}_n| \sim c\, n^{-5/2} \rho^{-n} n!,
\end{equation}
where $c > 0$ and $\rho$ is the radius of convergence of the (exponential) generating function $G(x) = \sum_{n\ge 0} g_n \frac{x^n}{n!}$ associated to the graph class under study.
In particular, we will restrict to so-called \emph{aperiodic} structures (see also Subsection~\ref{subs:subcr}), where the asymptotic relation \eqref{eqgn} holds for all $n$ and not only in some residue class.

The first result is an asymptotic estimate for both $|\mathcal{I}_n|$ and $|\mathcal{M}_n|$:

\begin{theorem}\label{thm:main1}
    Let $\mathcal{G}$ be a graph class satisfying proper aperiodicity and extended sub--criticality conditions (defined in Subsection~\ref{subs:subcr} and Section~\ref{sec:analysis-systems}), and let $\rho$ be the radius of convergence of the generating function $G(x)$ associated to $\mathcal{G}$.
    Then we have:
    \begin{equation*}
        |\mathcal{I}_n| \sim A_1\, n^{-5/2} \rho_1^{-n} n!
        \quad \mbox{and}\quad
        |\mathcal{M}_n| \sim A_2\, n^{-5/2} \rho_2^{-n} n!,
    \end{equation*}
    where $A_1, A_2, \rho_1, \rho_2$ are positive constants such that $0< \rho_1 < \rho$ and $0< \rho_2 < \rho$.
\end{theorem}

As a direct corollary, we obtain the following asymptotic estimates:

\begin{corollary}\label{cor:enum}
    Let $\mathcal{G}$ be as in Theorem~\ref{thm:main1} and let $AI_n$ be the average number of maximal independent sets in a graph of size $n$ in $\mathcal{G}$ and $AM_n$ be the average number of matchings in a graph of size $n$ in $\mathcal{G}$.
    Then it holds that
    \begin{equation*}
        AI_n = \frac{|\mathcal{I}_n|}{g_n} \sim C \cdot \alpha^{n}
        \quad \mbox{and}\quad
        AM_n = \frac{|\mathcal{M}_n|}{g_n} \sim D \cdot \beta^{n},
    \end{equation*}
    where $C, D, \alpha, \beta$ are positive constants and $\alpha$ and $\beta$ are larger than $1$.
\end{corollary}

Our second result concerns the distribution of the respective sizes of random maximal independent sets and matchings.
The following theorem shows that the limiting distribution follows a Central Limit Theorem with linear expectation and variance:

\begin{theorem}\label{thm:main2}
    Let $\mathcal{G}$ be a graph class satisfying proper aperiodicity and extended sub--criticality conditions (as in Theorem~\ref{thm:main1}).
    Furthermore, let $SI_n$ denote the size of a uniformly randomly chosen maximal independent set in $\mathcal{I}_n$ and $SM_n$ the size of a uniformly randomly chosen maximal matching in $\mathcal{M}_n$.
    Then,
    \begin{equation*}
        \begin{array}{lllllll}
            & \mathbb{E}[SI_n]&=&\mu n+ O(1),
            & \mathbb{V}\mathrm{ar}[SI_n]&=&\sigma_1^2 n+O(1), \\
            & \mathbb{E}[SM_n]&=&\lambda n+ O(1),
            & \mathbb{V}\mathrm{ar}[SM_n]&=&\sigma_2^2 n+O(1),
        \end{array}
    \end{equation*}
    for some constants $\mu,\lambda>0$ and $\sigma_1^2,\sigma_2^2 \ge 0$.
    Moreover, if $\sigma_1^2,\sigma_2^2 > 0$ then $SI_n$ and $MI_n$ satisfy a Central Limit Theorem:
    \begin{equation*}
        \frac{SI_n-\mathbb{E}[SI_n]}{\sqrt{\mathbb{V}\mathrm{ar}[SI_n]}} \stackrel{d}{\to} \mathrm{N(0,1)}
        \quad \mbox{and}\quad
        \frac{SM_n-\mathbb{E}[SM_n]}{\sqrt{\mathbb{V}\mathrm{ar}[SM_n]}} \stackrel{d}{\to} \mathrm{N(0,1)}.
    \end{equation*}
\end{theorem}

Apart from constants $A_1$, $A_2$ (in Theorem \ref{thm:main1}), and $C$, $D$ (in Corollary \ref{cor:enum}), all the other appearing constants can be computed explicitly to any degree of precision.
The following table lists some of them:

\begin{center}
    \begin{tabular}{||l||c|c||c|c||}
        \textit{Family}         & $\alpha$ & $\mu$ & $\beta$ & $\lambda$ \\
        \midrule
        Forests                 & 1.273864 & 0.463922 & 1.313080 & 0.357045 \\
        Cacti graphs            & 1.278323 & 0.431401 & 1.184091 & 0.346734 \\
        Series--parallel graphs  & 1.430394 & 0.269206 & 1.470167 & 0.318924
   \end{tabular}
\end{center}

The constants $\sigma_1^2,\sigma_2^2$ are much more difficult to calculate.
However, in all these concrete cases it can be shown that they are positive (see Section~\ref{S:applications}).

Let us mention that in \cite{MeirMoon88}, Meir and Moon obtained the estimate of Theorem~\ref{thm:main1} and the expectation in Theorem~\ref{thm:main2} for maximal independent sets in Cayley trees, plane trees and binary trees.
Our contribution generalises part of their work, providing a precise limiting distribution for the size of maximal independent sets in Cayley trees.

\medskip

Finally, let us briefly discuss the extremal versions of those problems.
In the literature, one can find two such directions.
One of them, started by Wilf \cite{wilf86} who was motivated by the design of an algorithm to compute the chromatic number, consists in characterising the extremal instances of a given family of graphs containing the maximum number of maximal independent sets (see \cite{griggs88} for connected graphs, \cite{sagan88} and \cite{wloch08} for trees, and \cite{GKSV06} then \cite{SV06} for graphs with a fixed number of cycles), as well as maximum independent sets (see \cite{zito91} and \cite{jou00}).
Furthermore, the maximum number of both maximal matchings \cite{gorska07} and maximum matchings \cite{heuberger11} have been treated.
The other direction consists in bounding the size of a maximum matching in a graph \cite{biedl04}.
However, the problems discussed in this paper seem to be of a different nature.
It is worth noticing that in \cite{biedl04}, the authors also give tight bounds on the size of a maximal matching in 3--connected planar graphs and in graphs with bounded maximum degree.

\paragraph{Structure of the paper.}

Section~\ref{sec:Prel} introduces the necessary background, namely the language of generating functions and how they apply to graph decompositions in terms of their connectivity, as well as the analytic concepts needed in the context of subcritical graph classes.
Later, in Section~\ref{sec:Eq} we obtain two implicit systems of functional equations respectively encoding maximal independent sets and maximal latchings in subcritical graph classes.
We then analyse them separatly using complex analytic tools in Subsection~\ref{sec:analysis-systems}.
In Section~\ref{S:applications}, we then apply our results to the families of Cayley trees, cacti graphs and series--parallel graphs.
The details concerning the computations in cacti graphs and series--parallel graphs are discussed in Appendix~\ref{app:appendix}.
Let us mention that in order to get the constants in the case of series--parallel graphs, we need to slightly reformulate the systems of equations following a grammar introduced in \cite{CFKS08}.
Finally, in Section~\ref{sec:remarks and further research} we discuss possible extensions of our work.

%
%

\section{Preliminaries}\label{sec:Prel}

\subsection{Generating functions}\label{sec:GF}

We follow the terminology from~\cite{fs05}.
A \emph{labelled combinatorial class} is a set $\mathcal{A}$ together with a size measure such that if $n\geq 0$, then the set of elements of size $n$ (denoted by $\mathcal{A}_n$) is finite.
Each element $a\in\mathcal{A}_n$ is built from $n$ atoms, which in our context (graph classes) are vertices with labels in the set $\{1,\dots,n\}$.
We always assume that the combinatorial classes of graphs we consider are stable under graph isomorphism, namely, $a\in \mathcal{A}$ if and only if all graphs $a'$ isomorphic to $a$ are also elements of $\mathcal{A}$.

In enumerative problems, it is often useful to use the exponential generating function (EGF, for short) associated to the labelled class $\mathcal{A}$, namely $A(x) := \sum_{n\geq 0} \frac{|\mathcal{A}_n|}{n!}x^n,\,\, [x^n]A(x) = \frac{|\mathcal{A}_n|}{n!}$.
In our setting, we use the (exponential) variable $x$ to encode vertices, and the (ordinary) variable $y$ to encode edges.
We also use other extra variables to encode different auxiliary parameters.

In this work, we deal with rooted and pointed combinatorial classes.
In particular, we can point the elements of a class $\mathcal{A}$ by distinguishing one of the atoms and discounting it, which means that we reduce the size function by $1$.
Since we assume that our combinatorial class is stable under graph isomorphism, this procedure can be performed by taking the atom with the largest label as the root.
The corresponding new pointed class will be denoted by $\mathcal{A}^\circ$.
Since every element of size $n$ in $\mathcal{A}$ corresponds to a unique element in $\mathcal{A}^\circ$, the term $x^n/n!$ in $A(x)$ is replaced by $x^{n-1}/(n-1)!$ in $A^\circ(x)$, i.e. $A^\circ(x) = A'(x).$

Similarly, we can consider a rooted structure $\mathcal{A}^\bullet$ by distinguishing one of the atoms without discounting it.
Since there are $n$ different ways of choosing an atom (for an element of size $n$), the corresponding term $x^n/n!$ in the generating function is replaced by $n x^n/n! = x^n/(n-1)!$, which leads to the relation $A^\bullet(x) = xA'(x)$.

Finally, we use in this paper the set, sequence and cycle constructions for combinatorial classes (see \cite{fs05} for the precise definitions).
Concerning notation, we write $-\log_{\geq k}(\frac{1}{1 - x}) = \sum_{r\geq k}\frac{x^r}{r} = -\log(\frac{1}{1 - x}) - \sum_{r = 1}^{k - 1}\frac{x^r}{r}$ for the generating function associated to the cycle construction such that that the number of components is greater or equal than $k$.
Similarly, the notation $\exp_{\geq k} = \sum_{r\geq k}\frac{x^r}{r!}$ will be associated to the restricted set construction.

\subsection{Graph decompositions}\label{subsec:graph-decomp}

A \emph{block} of a graph $G$ is a maximal 2-connected subgraph of $G$.
A graph class $\mathcal{G}$ is \emph{block--stable} if:
\begin{enumerate}
    \item it contains the single edge (which we denote by $e$),
    \item it satisfies that a connected graph $G$ belongs to $\mathcal{G}$ if and only if any one of its blocks is in $\mathcal{G}$.
\end{enumerate}
The class $\mathcal{G}$ is also said to be \emph{connected component--stable} when any graph $G$ is in $\mathcal{G}$ if and only if all connected components of $G$ belong to $\mathcal{G}$.
For a graph class $\mathcal{G}$, we denote by $\mathcal{C}$ and $\mathcal{B}$ the families of connected and $2$-connected graphs in $\mathcal{G}$, respectively.
We write $C(x)$ and $B(x)$ their corresponding EGF, and $G(x)$ for the EGF of $\mathcal{G}$.
In particular, if $\mathcal{G}$ is a block--stable and connected component--stable class of graphs, then the following combinatorial decomposition holds:
\begin{equation}\label{eq:graph-decomp}
	\mathcal{G} = \mathrm{Set}(\mathcal{C}),
	\qquad
	\mathcal{C}^{\bullet} = \bullet\times \mathrm{Set}(\mathcal{B}^{\circ}\circ \mathcal{C}^{\bullet}).
\end{equation}
The previous formulas read as follows: first, each graph in $\mathcal{G}$ is defined in terms of a set of connected graph in $\mathcal{C}$.
Secondly, a pointed connected graphs in $\mathcal{C}^{\bullet}$ can be decomposed as the root vertex, and a set of pointed blocks (the ones incident with the root vertex) where we substitute on each vertex a rooted connected graph.
See \cite{BerLaLe98,DrmotaBook,GoulJack83} for details.
These expressions translate into equations of EGFs in the following way:
\begin{equation*}
    G(x) = \exp(C(x)),
    \qquad
    C^{\bullet}(x) = x\exp(B^{\circ}(C^{\bullet}(x))).
\end{equation*}
We refer the readers to \cite{tutte1966} for further results on graph decompositions and connectivity on graphs.

\subsection{Asymptotics for subcritical graph classes}\label{subs:subcr}

We call a block--stable and vertex labelled graph class {\it subcritical} if $\eta B''(\eta) > 1$, where $\eta$ denotes the radius of convergence of $B(x)$.
In particular this condition is satisfied if
\begin{equation*}
	B''(x)\to \infty \qquad \mbox{ as $x\to \eta^-$.}
\end{equation*}
Cayley trees, cacti graphs, outerplanar graphs and series--parallel graphs are subcritical.
The main analytic property of subcritical graph classes is that they satisfy universal asymptotic behaviors, see \cite{Bernasconi2009, DrFuKaKrRu11, panagiotou2015, schben15, subgraphs}.
In our work, we just use the fact that the property $\eta B''(\eta) > 1$ ensures that the solution $C^\bullet(x)$ to the functional equation $C^{\bullet}(x)=x\exp(B^{\circ}(C^{\bullet}(x)))$ has a square--root type singularity at its radius of convergence $\rho$.

Recall that a domain dented at $x=\rho$ is a region of the complex plane of the form $\Delta(\phi,R)=\{x \in \mathbb{C} : x \neq \rho, |x| < R, |\mathrm{Arg}(x - \rho)| > \phi\}$.
With the previous assumptions $C^\bullet(x)$ has a local expansion in a domain dented at $x=\rho$ of the form:
\begin{equation}\label{eqCsing}
    C^\bullet(x) = xC'(x) = \underline{c}_0 + \underline{c}_1 \left({1- \frac x\rho}\right)^{1/2}
    + \underline{c}_2 \left( 1 - \frac x\rho \right)
    + \underline{c}_3 \left( 1 - \frac x\rho \right)^{3/2} + \cdots,
\end{equation}
where $\rho$ is given by the equation $\rho = \underline{c}_0 e^{-B'(\underline{c}_0)}$ and $0 <\underline{c}_0 = C^\bullet(\rho) < \eta$ is given by the equation $\underline{c}_0 B''(\underline{c}_0) = 1$. Furthermore $\underline{c}_1 < 0$.
Note that the singular behavior of $B(x)$ at its radius of convergence $\eta$ is irrelevant for the singular behavior of $C^\bullet(x) = xC'(x)$, we only make use of the (analytic) behavior of $B'(x) = B^\circ(x) $ around $x = \underline{c}_0 < \eta$.

Assuming that the class $\mathcal{G}$ is also connected component--stable, then it follws from \eqref{eqCsing} that $C(x)$ and $G(x) = e^{C(x)}$ have the following singular behavior in a domain dented at $\rho$:
\begin{equation*}
    \begin{array}{cc}
        & C(x) =  c_0 +  c_2 \left( 1 - \frac x\rho \right) +  c_3 \left( 1 - \frac x\rho \right)^{3/2} + \cdots, \\
        & G(x) =  g_0 +  g_2 \left( 1 - \frac x\rho \right) +  g_3 \left( 1 - \frac x\rho \right)^{3/2} + \cdots,
    \end{array}
\end{equation*}
where $c_3$ and $g_3$ are positive.

If we further assume that $B'(x) = B^{\circ}(x)$ is \emph{aperiodic} in the sense that $B^{\circ}(x)$ cannot be represented in the
form $B^{\circ}(x) = x^e b(x^d)$ for some $e\geq 0$ and $d> 1$. Then it follows that $x=\rho$ is the only singularity on the circle of convergence
$|x|=\rho$. This is actually satisfied in all the cases under our study here, and hence for proper positive constants $c',c''$ it follows that (see for instance \cite{fs05}):
\begin{equation*}
    |\mathcal{C}_n| = n!\, [x^n]\,  C(x) \sim c' \cdot n^{-5/2} \rho^{-n} n!
    \quad \mbox{and} \quad
    |\mathcal{G}_n| = n!\, [x^n]\,  C(x) \sim c''\cdot n^{-5/2} \rho^{-n} n!.
\end{equation*}

%
%

\section{Counting in block--stable graph classes}\label{sec:Eq}

In this section, we consider block--stable vertex labelled graph classes and set up functional equations for counting maximal independent subsets and maximal matchings.
We recall our notations for graph classes: $\mathcal{B}$ and $\mathcal{C}$ denote the families of 2--connected blocks and connected graphs, respectively, in a block--stable graph class $\mathcal{G}$.

\subsection{Maximal independent sets in block--stable graph classes}\label{ssec:mis}

A {\em coloured block} is a pair $(I,b)$ consisting of a block $b\in \mathcal{B}$ together with a distinguished independent set $I$ of $b$.
Note that $I$ can be any independent set of $b$, i.e. not necessarily maximal.
Let $B(x,y_0,y_1,y_2)$ be the EGF counting coloured blocks, where
\begin{itemize}
    \item $x$ marks vertices,
    \item $y_0$ counts vertices in $I$,
    \item $y_1$ counts vertices adjacent to vertices in $I$ (namely, vertices at distance one from $I$),
    \item $y_2$ counts the rest of the vertices (vertices at distance at least or equal than two from $I$).
\end{itemize}
Similarly, a \emph{pointed coloured block} is a pair $(I, b^{\circ})$ consisting of a pointed block $b^{\circ}\in \mathcal{B}^{\circ}$ together with a distinguished independent set $I$ of $b$.
Let $B_i := B_i(x,y_0,y_1,y_2)$ be the EGFs counting pointed coloured blocks, where the pointed vertex is at distance {\em exactly} $i$ from $I$ for $i\in \{0, 1\}$, and at distance at least two for $i=2$.
So that we have:
\begin{equation*}
    B_i= \frac{1}{x}\cdot \frac{\partial B}{\partial y_i}, \quad \mbox{for } i\in \{0, 1, 2\}.
\end{equation*}
A {\em coloured graph} $(J,g)$ is a pair consisting of a connected graph $g\in \mathcal{C}$ and of a \emph{maximal} independent set $J$ of $g$.
Let $C:=C(x, y_0, y_1)$ be the EGF counting coloured graphs, where $y_0$ and $y_1$ encode vertices in $J$ and vertices at distance one from $J$, respectively.
For $i\in \{0, 1\}$, let $C_i = C_i(x,y_0,y_1)$ be the EGFs enumerating pointed coloured--graphs, for which the pointed vertex is at distance {\em exactly} $i$ from $J$.
These generating functions are given by
\begin{equation*}\label{eq:C_i_mis}
    C_i = \frac{1}{x}\cdot \frac{\partial C}{\partial y_i}, \quad \mbox{for } i\in \{0,1\}.
\end{equation*}
We finally need an auxiliary class.
A \emph{special pointed coloured graph} is a pair $(J,g^{\circ})$ where $J$ is an independent set of $g$ not including the pointed vertex, such that the pointed vertex is at distance at least two from $J$ and which becomes maximal when adding the pointed vertex to $J$.
Equivalently, a special pointed coloured graph is obtained from a coloured graph pointed at a vertex in $J$ by removing it from $J$.
We denote the corresponding generating function by $C_2(x, y_0,y_1)$.
Observe that given a coloured graph $(J,g)$, the independent set $J$ together with the vertices of $g$ at distance one from $J$ form a partition of the vertices of $g$.
Hence, the following equalities hold:
\begin{equation*}\label{eq:con-mis}
    \frac{\partial C}{\partial x} = \frac{y_0}{x}\frac{\partial C}{\partial y_0} + \frac{y_1}{x}\frac{\partial C}{\partial y_1} = y_0 C_0 +y_1 C_1.
\end{equation*}
We also have that $G(x,y_0,y_1) = \exp( C(x,y_0,y_1))$, where $G(x,y_0,y_1)$ denotes the corresponding generating function of coloured graphs in $\mathcal{G}$.

The following lemma describes connected structures in terms of their block--decomposition (see Figure \ref{fig:mis_decompositon} for an example).
Thus if we know $B(x,y_0,y_1,y_2)$ (or just $B_i(x,y_0,y_1,y_2)$ for each $i\in \{0,1,2\}$), then we can determine $\frac{\partial C}{\partial x}(x,y_0,y_1)$ and consequently $C(x,y_0,y_1)$ and $G(x,y_0,y_1)$.

\begin{figure}[h!]
\centering
    \includegraphics[scale=1.29]{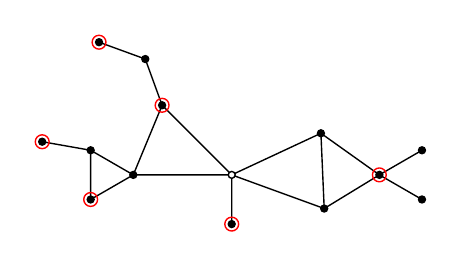}
    \includegraphics[scale=1.29]{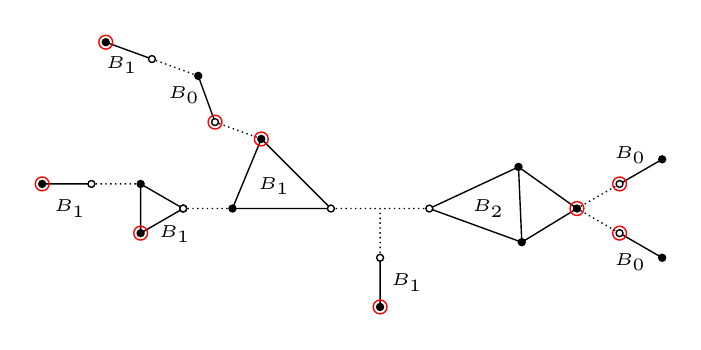}
    \caption{
        Left is a pointed coloured graph with independent set $J$. Vertices of $J$ are circled in red, the pointed vertex is painted as a white circle.
        Right is its block--decomposition.
        Pointed vertices are coloured (again) in white.
    }
    \label{fig:mis_decompositon}
\end{figure}

\begin{lemma}\label{lem:blocks_mis}
    With the above notations, the following system of equations holds:
    \begin{equation}\label{eqLe4syst}
        \begin{array}{ll}
            & C_0 = \exp(B_0(x, y_0C_0, y_1(C_1 + C_2), y_1C_1)), \\
            & C_1 = \left(\exp(B_1(x, y_0C_0, y_1(C_1 + C_2), y_1C_1))-1\right)\cdot C_2, \\
            & C_2 = \exp(B_2(x, y_0C_0, y_1(C_1 + C_2), y_1C_1)).
        \end{array}
    \end{equation}
\end{lemma}

\begin{proof}
    The proof is based on a refinement of Equation~\eqref{eq:graph-decomp} expressing the EGF of a family of pointed connected graphs in terms of a set of pointed blocks in which we substitute each vertex by a pointed connected graph.
    In what follows, we decompose our initial pointed coloured graph into a set of pointed coloured blocks, and we substitute each of their vertices (of the three different types) by different families of pointed coloured graphs.

    Let us start by proving that the implicit equation defining $C_0$ holds.
    Let $(I, g^{\circ})$ be a pointed coloured graph whose pointed vertex is in $I$.
    We decompose $g^\circ$ according to the blocks that are incident with its pointed vertex.
    Observe that the pointed vertex of $g^{\circ}$ determines a set of pointed coloured blocks $(J_i, b_i^{\circ})$ (with $i=1,\dots, k$ for a certain $k$): each $b_i^{\circ}$ is a maximal 2-connected graph incident with the pointed vertex of $g^\circ$, and $J_i$ is the restriction of $I$ into $J_i$.
    Observe that $J_i$ is an independent set but \emph{not necessarily maximal}.
    Without loss of generality, let us now fix $j\in \{1,\dots, k\}$ and analyse the structure of the pair $(J_j,b_j^{\circ})$ with respect to the whole pair $(I, g^{\circ})$.
    First, to every vertex of $b_j^{\circ}$ in $J_j$ must be attached a coloured graph $(L, h^{\circ})$ whose root is in $L$. That is, a coloured graph counted by $C_0$.
    In terms of generating functions, this translates into the substitution of $y_0$ by $y_0C_0$.
    Second, to each vertex of $b_j^{\circ}$ at distance one from $J_j$, the root of the pointed coloured graph $(L, h^{\circ})$ attached to it can be at distances either one or more from $L$.
    This translates into the substitution of $y_1$ by $y_1(C_1 + C_2)$.
    Finally, if a vertex of $b_j^{\circ}$ is at distance at least two from $J_j$, then the root of the coloured graph $(L, h^{\circ})$ attached to it must be at distance one from $L$, as we need to extend the independent set to one that is maximal.
    This translates into the substitution of $y_2$ by $y_1C_1$ and the first equation of~\eqref{eqLe4syst} holds.

    The implicit equation satisfied by $C_2$ is obtained following the exact same arguments as for $C_0$.

    Let us finally discuss the equation for $C_1$.
    Assume that $(I,g^{\circ})$ is a pointed coloured graph and that $(J_i, b_i^{\circ})$ (for $i=1,\dots,k$) are the pointed coloured blocks incident with the pointed vertex of $g^{\circ}$.
    In particular, for each $i\in \{1,\ldots, k\}$ the pointed vertex of $b_i^{\circ}$ is at distance either one or at least two from $J_i$.
    Nevertheless, observe that there exists at least one of the pointed-blocks $(J_j, b_j^{\circ})$ whose pointed vertex is at distance one from $J_j$.
    This concludes the proof, as it gives us that:
    \begin{align*}
        C_1 & = \exp_{\geq 1}(B_1(x, y_0C_0, y_1(C_1 + C_2), y_1C_1))\cdot \exp(B_2(x, y_0C_0, y_1(C_1 + C_2), y_1C_1)) \\
            & = \left(\exp(B_1(x, y_0C_0, y_1(C_1 + C_2), y_1C_1))-1\right)\cdot C_2,
    \end{align*}
    where the last equality is obtained using the equation for $C_2$.
\end{proof}

\subsection{Maximal matchings in block--stable classes of graphs}\label{subsec:max-match-block}

In this subsection we deal with the case of maximal matchings.
Most of the definitions and concepts are the natural analogues of the ones developed in the case of maximal independent sets.
Hence, we will skip unnecessary repetitions.

A \emph{matched block} is a triple $(I,M,b)$ with a block $b\in \mathcal{B}$, a matching $M$ in $b$, and an independent set $I$ of $b$, and where no element of $I$ is incident to an edge in $M$.
In other words, we split the set of vertices of $b$ in three disjoint subsets: vertices in $I$, matched vertices, and the rest.
The former vertices will be called \emph{marginal} vertices.

A \emph{pointed} matched block is a triple $(I,M,b^{\circ})$, where $b^{\circ}\in \mathcal{B}^{\circ}$ and $M$ and $I$ are respectively a matching and an independent set of $b$, and where again no element of $I$ is incident to an edge in $M$.
Let $\bar B(x, z_0, z_1, z_2)$ be the EGF counting matched blocks, where the variable $x$ marks vertices, and $z_0$, $z_1$ and $z_2$ mark vertices in $I$, vertices matched by $M$ and marginal vertices respectively.
Observe that the exponent of the variable $z_1$ in $B(x, z_0, z_1, z_2)$ is always an even number.
This is due to the fact that $z_1$ counts pairs of vertices.
This is also true for the forthcoming EGF.
For $i\in \{0, 1, 2\}$ let $\bar B_i = \bar B_i(x, z_0, z_1, z_2)$ be the generating function counting pointed matched blocks where the pointed vertex is either in $I$, is incident with $M$ or pointed at a marginal vertex.
In particular,
\begin{equation*}
    \bar B_i = \frac{1}{x}\cdot \frac{\partial \bar B}{\partial z_i}, \quad \mbox{for } i\in\{0,1,2\}.
\end{equation*}

A \emph{matched graph} is a triple $(I,M,g)$ consisting of a connected graph $g$ in $\mathcal{C}\subseteq\mathcal{G}$, a matching $M$ of $g$, and an independent set $I$ not incident with $M$.
Similarly, a \emph{pointed} matched graph is a triple $(M,I,g^{\circ})$ where now $g^{\circ}$ is a pointed graph.
Let $\bar C(x, z_0, z_1, z_2)$ be the exponential generating function counting matched graphs, where $x$, $z_0$, $z_1$ and $z_2$ respectively mark vertices, vertices incident with $I$, vertices incident with $M$, and marginal vertices.
Notice that when $z_2 = 0$, $\bar C := \bar C(x,z_0,z_1) = \bar C(x,z_0,z_1,0)$ encodes matched graphs where $M$ is a maximal matching.
For each $i\in \{0, 1, 2\}$, let us define the following generating function
\begin{equation*}
    \bar C_i=\bar C_i(x, z_0, z_1) = \frac{1}{x}\cdot \frac{\partial \bar C}{\partial z_i}(x, z_0, z_1, 0).
\end{equation*}
Observe then that $\bar C_0$ counts pointed matched graphs, where the matching is maximal and the pointed vertex belongs to the independent set, $\bar C_1$ counts pointed matched graphs, where the matching is maximal and the pointed vertex belongs to the matching, whereas $\bar C_2$ counts pointed matched graphs, where the matching is not necessarily maximal and the pointed vertex is marginal.
In the latter case, the matching is maximal except for possibly the pointed vertex, which might be unmatched and adjacent to other unmatched vertices.
In particular, this implies that the generating function of pairs of connected graphs endowed with a maximal matchings is given by:
\begin{equation*}\label{eqCxrep}
    \frac{\partial{\bar C}}{{\partial x}} = z_0\bar C_0 + z_1\bar C_1.
\end{equation*}
The main idea behind the use of these generating functions is that vertices in the independent set $I$ play the role of vertices that will not be matched in the block decomposition.
In particular, we exploit the knowledge of $I$ in order to ensure that the matching cannot be extended.
On the other hand, the set of marginal vertices will be matched by attached pointed matched graphs in the decomposition.
See Figure \ref{fig:match_decompositon} for an example of the decomposition of a pointed matched graph in terms of its blocks and the different types of vertices.

The following lemma relates all the previous generating functions.
Note that the generating functions $\bar C(x,z_0,z_1,0)$ and $\bar G(x,z_0,z_1) = \exp(\bar C(x,z_0,z_1))$ directly follow from the solution of the next system.

\begin{figure}[h!]
	\centering
    \includegraphics[scale=1.1]{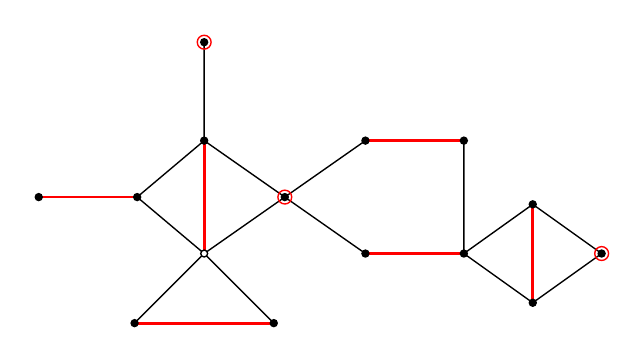}
    \includegraphics[scale=1.1]{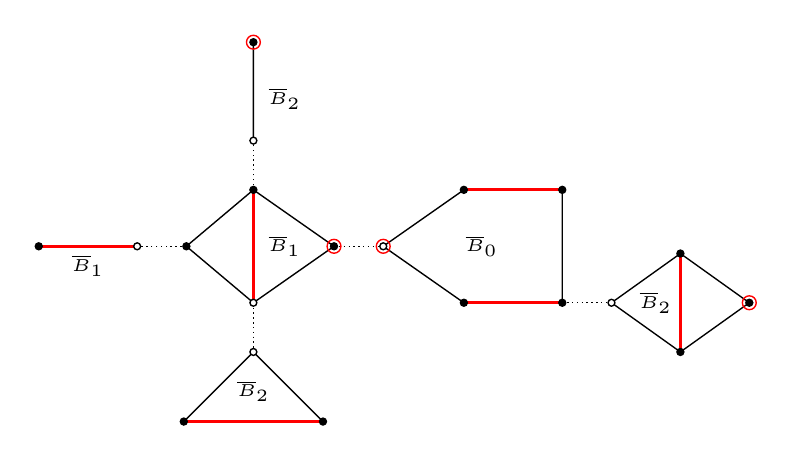}
    \caption{
        On the left there is a connected graph pointed at a vertex (in white) incident to a maximal matching (induced by the edges in red), which is counted by the EGF $\bar C_1(x,z_0,z_1,z_2)$.
        The vertices belonging to the independent set are circled in red.
        Right it is drawn its decomposition into rooted blocks.
    }
    \label{fig:match_decompositon}
\end{figure}

\begin{lemma}\label{lem:blocks_mm}
	The following equalities hold:
	\begin{equation}\label{eqLe5syst}
    	\begin{array}{ll}
        	\bar C_0 &= \exp(\bar B_0(x, z_0\bar C_0, z_1\bar C_2, z_1\bar C_1)), \\
        	\bar C_1 &= \bar C_2\,\bar B_1(x, z_0\bar C_0, z_1\bar C_2, z_1\bar C_1), \\
        	\bar C_2 &= \exp(\bar B_2(x, z_0\bar C_0, z_1\bar C_2, z_1\bar C_1)).
    	\end{array}
	\end{equation}
\end{lemma}

\begin{proof}
    We proceed similarly to the proof of Lemma \ref{lem:blocks_mis}.
	To that end, let $(I,M,g^{\circ})$ be a pointed matched graph, with pointed vertex $v$.

	Suppose first that $v\in I$, i.e. the case counted by $\bar C_0$.
	It is the pointed vertex of a set of adjacent pointed blocks $(I_j,M_j,b^{\circ}_j)$, in which $v\in I_j$, and is not adjacent to any other pointed block.
	This means that all the pointed blocks adjacent to $v$ are counted by $\bar B_0$.
	Furthermore, to each vertex of those blocks incident to $I$ is attached a matched graph $(I',M',h^{\circ})$ pointed at a vertex incident to $I'$.
	Similarly, to each vertex incident to $M$ is attached a matched graph pointed at a marginal vertex, as else two edges of the resulting matching will be incident.
	And to each marginal vertex is attached a matched graph $(I'',M'',f^{\circ})$ pointed at a vertex incident to $M''$. 
	This translates into the substitutions of the variables $z_0$, $z_1$ and $z_2$ in $\bar B_0$ by $z_0\bar C_0$, $z_1\bar C_2$ and $z_2\bar C_1$, respectively.
	The first equation then follows.
	
	Suppose next that the vertex $v$ is matched, i.e. the case counted by $\bar C_1$.
	Then the edge of $M$ incident with $v$ must belong to a single pointed block, whose pointed vertex (identified with $v$) is incident to an edge of the respective matching.
	So that one can only attach this particular block to $v$, together with any connected matched graph pointed at a marginal vertex.
	This gives the second equation in~\eqref{eqLe5syst}.
	Finally, the equation for $\bar C_2$ is obtained following similar arguments as for $\bar C_0$.
\end{proof}

%
%

\section{Asymptotic analysis. Proofs of Theorems \ref{thm:main1} and \ref{thm:main2}}\label{sec:analysis-systems}

In this section, we study the analytic properties of the solutions of Systems~\eqref{eqLe4syst} and~\eqref{eqLe5syst}, provided that the functions $B_0$, $B_1$ and $B_2$ (resp. $\bar B_0$, $\bar B_1$ and $\bar B_2$) all behave in a \textit{proper way} that is similar to the behaviour of $B(x)$ (resp. $\bar B(x)$) in the case of block--stable subcritical graph classes.
Under the hypotheses we assume, it will be rather direct to prove both Theorems~\ref{thm:main1} and~\ref{thm:main2}.
We will fully discuss the case of maximal independent sets.
A similar argument would apply for the study of maximal matchings.
We will only explain the differences at the end of the section.

First, note that the functions $B_i(x,y_0,y_1,y_2)$ are actually functions in three variables since a monomial $x^n y_0^{k_0} y_1^{k_1} y_2^{k_2}$ can only appear if $k_0 + k_1 + k_2 = n$, that is, we have $B_i(x,y_0,y_1,y_2) = B_i(1,xy_0,xy_1,xy_2)$ or equivalently $B_i(x,y_0,y_1,y_2) = B_i(xy_2,y_0/y_2,y_1/y_2,1)$.
However, it is more convenient to work with all four variables $x$, $y_0$, $y_1$, and $y_2$.

Now, if $y_0$, $y_1$ and $y_2$ are positive real numbers, then the function $x\mapsto B(x, y_0, y_1, y_2)$ is a power series with non--negative coefficients.
Hence the radius of convergence of this function coincides with its dominant singularity in $x$.
We will denote this radius of convergence by $R(y_0,y_1,y_2)$.
Similarly, for the three solution functions $C_0(x,y_0,y_1)$, $C_1(x,y_0,y_1)$ and $C_2(x,y_0,y_1)$ of the system~\eqref{eqLe4syst}, we denote by $\rho_i(y_0,y_1)$ ($i = 0,1,2$) the radius of convergence of $C_i(x,y_0,y_1)$ with respect to $x$ when $y_0$ and $y_1$ are (fixed) positive real numbers.

We now introduce the so-called \textbf{extended subcriticallity conditions}:

\begin{itemize}
	\item[]
	\begin{itemize}
	\item[(C1)] \textit{For any triple $(y_0,y_1,y_2)$ of positive real numbers, the function $R(y_0,y_1,y_2)$ can be extended analytically for a sufficiently small neighbourhood.}
\end{itemize}
\end{itemize}
\textit{For all positive real numbers $y_0$, $y_1$ and $y_2$, there are  $i_1,i_2\in\{0,1,2\}$ with}:
\begin{itemize}
	\item[]
	\begin{inparaenum}
	\item[(C2)] $\ds\frac{\partial^2 B}{\partial y_{i_1}^2}(0,y_0,y_1,y_2) = 0$
     	\qquad \mbox{and} \qquad
    \item[(C3)] $\ds\lim_{x \to R(y_0,y_1,y_2)^-} \frac{\partial^2 B}{\partial y_{i_1}^2}(x,y_0,y_1,y_2) = \infty$.
	\end{inparaenum}
\end{itemize}

We define similarly $\bar R(z_0,z_1,z_2)$ for the radius of convergence of the function $x\mapsto \bar B(x,z_0,z_1,z_2)$.
And introduce three analogue analytic conditions by replacing $B$ with $\bar B$, $R$ with $\bar R$ and $y_i$ with $z_i$ for $i=0,1,2$ in (C1) -- (C3).

We say that a graph family is \emph{extended subcritical} with respect to the parameters \emph{maximal independent sets} and
\emph{maximal matchings} if the above three conditions (C1) -- (C3) hold.
This is a quite technical condition and seems to be difficult to check (even more difficult than the condition $\eta B''(\eta) > 1$ for subcritical graph families).
However, we will see that in our applications the functions $B_i(x, y_0, y_1, y_2)$ (and $\bar B_i(x, y_0, y_1, y_2)$) are either polynomials or functions with a dominating squareroot singularity of the following form:
\begin{equation*}
	B_i(x, y_0, y_1, y_2) = G_i(x, y_0, y_1, y_2) - H_i(x, y_0, y_1, y_2)\sqrt{ 1- \frac x{R(y_0,y_1,y_2)}},
\end{equation*}
with $H_i(x, y_0, y_1, y_2)\ne 0$.
These kind of singularities appear naturally if -- as in our cases -- the functions $B_i(x,y_0,y_1,y_2)$ satisfy a positive and strongly connected system of equation (see~\cite[Theorem 2.33]{DrmotaBook}).
So that if we have a squareroot singularity of that form, then the rather restrictive condition (C3) follows immediately.

The next lemma introduces the most technical part of the proof, providing a singular expansion of each of the functions $\{C_i(x,y_1,y_2)\}$ for $i=0,1,2$ under conditions (C1), (C2) and (C3).

\begin{lemma}\label{lem:Le6}
	Assume that the extended subcriticallity conditions are satisfied.
	Then the solutions $C_0$, $C_1$ and $C_2$ of system \eqref{eqLe4syst} have the property that the functions $\rho_i(y_0,y_1)$ ($i = 0,1,2$) coincide and extend to an analytic function $\rho_1(y_0,y_1)$, for a sufficiently small neighbourhood around the positive real numbers.
	Moreover, the dominant singularity is of square--root type at $\rho_1(y_0,y_1)$ and we have a local representation of the form
	\begin{equation}\label{eqlem:Le62}
    	C_i(x,y_0,y_1) = g_{i}(x,y_0,y_1) - h_{i}(x,y_0,y_1)\left({ 1- \frac x{\rho_1(y_0,y_1)} }\right)^{1/2},
	\end{equation}
	where the functions $g_{i}(x,y_0,y_1)$ and $h_{i}(x,y_0,y_1)$ are analytic for $x$ close to $\rho_1(y_0,y_1)$ and $y_0,y_1$ close to the positive real axis.
	Furthermore $h_{i}(\rho_1(y_0,y_1),y_0,y_1) > 0$, for positive reals $y_0$ and $y_1$.
\end{lemma}

\begin{proof}
	Before using the extended subcriticality conditions, we recall some basic facts on (positive) systems of functional equations that are taken from~\cite[Theorem 2.33]{DrmotaBook}.
	Let $f$, $g$ and $h$ be power series with non--negative coefficients.
    Suppose that we have a system of three equations of the form
	\begin{equation}\label{eq:impl.system}
    	\begin{array}{ll}
        	A_1 &= f(x,A_1,A_2,A_3), \\
        	A_2 &= g(x,A_1,A_2,A_3), \\
        	A_3 &= h(x,A_1,A_2,A_3),
    	\end{array}
	\end{equation}
	in unknown functions $A_1 = A_1(x)$, $A_2 = A_2(x)$, $A_3 = A_3(x)$.
	We also assume that the system is strongly connected (namely, no subsystem can be solved before solving the whole system).
    We set
	\begin{equation*}
	    \Delta = \left|
    	\begin{array}{ccc}
        	1- f_{A_1} & - f_{A_2} & - f_{A_3} \\
            -g_{A_1} & 1- g_{A_2} & - g_{A_3} \\
            -h_{A_1} & - h_{A_2} & 1- h_{A_3}
    	\end{array} \right|
    \end{equation*}
	the functional determinant of the system $\{A_1 - f = 0,\, A_2 - g = 0,\, A_3 - h = 0\}$.
	Let $r$ be the spectral radius of the Jacobian matrix of the right hand--side of System~\eqref{eq:impl.system}.
	Note that $r = 1$ implies that $\Delta = 0$.
	
	We also assume that there is a unique non--negative solution $A_1(0)$, $A_2(0)$ and $A_3(0)$ with the property that $r < 1$, which also shows that $\Delta \ne 0$.
	Furthermore by monotonicity of the spectral radius, it follows that $f_{A_1}(0,A_1(0),A_2(0),A_3(0)) < 1$, $g_{A_2}(0,A_1(0),A_2(0),A_3(0)) < 1$, and $h_{A_3}(0,A_1(0),A_2(0),A_3(0)) < 1$.
	By the property $r<1$ and iteration, the solutions for $x=0$ extend to the three power series solutions $A_1(x)$, $A_2(x)$ and $A_3(x)$ with non--negative coefficients and (due to strong connectedness) a common positive radius of convergence $\bar \rho$.
	Also, by~\cite[Theorem 2.33]{DrmotaBook} this condition is determined by the fact that $r = 1$, provided that we are still working within the region of convergence of $f$, $g$, and $h$.
	This means that the radii of convergence
	$\rho_f(a_1,a_2,a_3)$, $\rho_g(a_1,a_2,a_3)$ and $\rho_h(a_1,a_2,a_3)$
	of the mappings
	$x\to f(x,a_1,a_2,a_3)$, $x\to g(x,a_1,a_2,a_3)$ and  $x\to h(x,a_1,a_2,a_3)$ (where $a_1,a_2,a_3$ are positive)
	are big enough in the sense that
	$\rho_f(A_1(\bar \rho),A_2(\bar\rho), A_3(\bar\rho)) > \bar\rho$, $\rho_g(A_1(\bar \rho),A_2(\bar\rho), A_3(\bar\rho)) > \bar\rho$ and $\rho_h(A_1(\bar \rho),A_2(\bar\rho), A_3(\bar\rho)) > \bar\rho$.
		
	The condition $r=1$ can be witnessed by the condition $\Delta = 0$, or equivalently by the condition
	\begin{equation*}\label{eqD0}
    	\frac{f_{A_2} g_{A_3} h_{A_1} + f_{A_3} g_{A_1} h_{A_1} }{(1-f_{A_1})(1-g_{A_2})(1-h_{A_3})} + \frac{g_{A_3} h_{A_2}}{(1-g_{A_2})(1-h_{A_3})} + \frac{f_{A_3} h_{A_1}}{(1-f_{A_1})(1-h_{A_3})} + \frac{f_{A_2} g_{A_1}}{(1-f_{A_1})(1-g_{A_2})} = 1.
	\end{equation*}
	Note that the left hand--side is an increasing and continuous function in $x$.
	Thus, if we assume that $r<1$ for $x=0$ then the left hand--side is smaller than $1$ for $x=0$.
	Furthermore, if we have either
	$f_{A_1}(x,a_1,a_2,a_3) \to \infty$ if $x\to \rho_f(a_1,a_2,a_3)$, or
	$g_{A_2}(x,a_1,a_2,a_3)\to\infty$ if $x\to \rho_g(a_1,a_2,a_3)$, or
	$h_{A_3}(x,a_1,a_2,a_3) \to \infty$ if  $x\to \rho_h(a_1,a_2,a_3)$
	then the left hand--side will eventually reach the value one (it would be sufficient to replace $\infty$ by any limit $\ge 1$).
	Thus, in order to find $\bar \rho$ we just have to find the $x$ for which the left hand--side hits the value one.
	By assumption we reach the value $1$ inside the region of convergence of $f$, $g$ and $h$.
	Hence, by~\cite[Theorem 2.33]{DrmotaBook} it follows that the solution functions $A_1(x)$, $A_2(x)$ and $A_3(x)$ have a square--root type singularity with an expansion of the claimed form at $x = \bar\rho$.
	Of course everything can be done, too, if the system of equations has some additional parameters, for example $y_0,y_1,y_2$, as in our application cases.

    Let us now move to the special situation of System~\eqref{eqLe4syst}.
    In this situation, all the above assumptions (positivity, strong connectedness, $\dots$) are satisfied.
	Now let us also observe that $\partial^2 B/\partial y_0^2 \to \infty$ (when $x \to R(y_0,y_1,y_2)^-$) implies that $f_{A_1}\to \infty$ as well, since
    \begin{equation}\label{eq:mis-f}
	   f = \exp(B_0(x,y_0 A_1,y_1(A_2+A_3),y_1 A_2))
    \end{equation}
	and $B_0 = (1/x)(\partial B/\partial y_0)$ (note the two different meanings of $y_0$: in this second equation we mean the derivative with respect to the second variable of $B_0(x,y_0,y_1,y_2)$).
	Similar observations hold for $g_{A_2}$ and $h_{A_3}$.
	Thus, it is clear that \eqref{eqD0} is satisfied inside the region of convergence of $f$, $g$ and $h$, and hence we are done.
\end{proof}

We now show that under the hypotheses of Lemma~\ref{lem:Le6}, one can deduce our main results Theorem~\ref{thm:main1} and Theorem~\ref{thm:main2}.

\begin{proof}[Proof of Theorem~\ref{thm:main1} and Theorem~\ref{thm:main2} for MIS]
    From \eqref{eqlem:Le62} and \eqref{eq:con-mis}, it follows that $C(x,y_0,y_1)$ can be represented, for $x\to \rho_1(y_1,y_2)$, as
    \begin{equation*}
        C(x,y_0,y_1) = c_{0}(y_0,y_1) + c_{2}(y_0,y_1)\left(1 - \frac x{\rho_1(y_0,y_1)}\right) + c_{3}(y_0,y_1) \left(1 - \frac x{\rho_1(y_0,y_1)}\right)^{3/2} +\cdots ,
    \end{equation*}
    where $c_3(y_0,y_1)> 0$ for positive real numbers $y_0$ and $y_1$.
    Thus, if we set $y_0 = y_1 = 1$ and $\rho_1(1,1) = \rho_1$, then for $x\to \rho_1$ we have:
    \begin{equation*}
        C(x,1,1) = c_{0}(1,1) + c_{2}(1,1)\left(1 - \frac x{\rho_1}\right) + c_{3}(1,1)\left(1 - \frac x{\rho_1}\right)^{3/2} + \cdots,
    \end{equation*}
    and consequently
    \begin{align*}
        G(x,1,1) = \sum_{n\ge 0} |\mathcal{I}_n| \frac{x^n}{n!}
        & = \exp(C(x,1,1)) \\
        & = g_{0}(1,1) + g_{2}(1,1)\left(1 - \frac x{\rho_1}\right) + g_{3}(1,1)\left(1 - \frac x{\rho_1}\right)^{3/2} + \cdots.
    \end{align*}
    This directly implies Theorem~\ref{thm:main1} for the case of maximal independent sets by standard singularity analysis (see \cite{fs05}).
    We just have to observe that $y_1 = \rho_1 = \rho_1(1,1)$ is the only singularity on the circle of convergence.
    However, this follows from the fact that there exists graphs of all sizes $n\ge 1$.

    Finally, if we set $y_1 = 1$, then we have that $G(x,y_0,1)=\sum_{n\ge 0}  \mathbb{E}[y_0^{SI_n}]\, | \mathcal{I}_n |\, \frac{x^n}{n!}$ is equal to
    \begin{align*}
        \exp(C(x,y_0,1)) =  g_{0}(y_0,1) + g_{2}(y_0,1) \left( 1- \frac x{\rho_1(y_0,1)} \right) + g_{3}(y_0,1) \left( 1- \frac x{\rho_1(y_0,1)} \right)^{3/2}+ \cdots.
    \end{align*}
    So a direct application of \cite[Theorem 2.35]{DrmotaBook} implies a central limit theorem of the proposed form, as well as the asymptotic expansions for the expectation and variance.
    This proves Theorem~\ref{thm:main2} for the case of maximal independent sets.
\end{proof}

\paragraph{Observation.}

The same arguments can be applied to the study of maximal matchings.
Observe that the analogue of Equation~\eqref{eq:mis-f} satisfies the same properties, so again in this setting we can apply Lemma~\ref{lem:Le6}.
Similarly, one can use the very same arguments as in the previous proof to obtain the conclusions in Theorems~\ref{thm:main1} and~\ref{thm:main2} for maximal matchings.

The only difference is that in our setting, matchings are counted in terms of vertices, so the number of matchings will be always an even number.
This fact does not affect the result in Theorem 1. However, in terms of computations, one needs to be careful concerning Theorem 3.
Observe that $\bar{C}(x,z_1)$ is written in terms of $z_1^2$, as matchings are counted in terms of vertices.
Let $\gamma(z_1):=\rho(z_1^2)$ be the corresponding radius of convergence.
Then if we want to study the distribution of the number of edges (instead of vertices), we have to study the associated EGF $\bar{\bar{C}}(x,z_1)=\bar{C}(x,z_1^{1/2})$, which has radius of convergence $\rho(z_1)$.
And we can apply the Quasi--Powers Theorem for $\bar{\bar{C}}(x,z_1)$, so that the corresponding expectation is asympotically linear with multiplicative constant
\begin{equation*}
	- \frac{\rho'(1)}{\rho(1)} = -\frac{1}{2}\frac{\gamma('1)}{\gamma(1)}.
\end{equation*}

%
%

\section{Applications}\label{S:applications}

This section is devoted to the verification of the three conditions (C1) -- (C3) for particular graph families.
Namely, in the cases of Cayley trees, cacti graphs and series--parallel graphs.
Notice that Condition (C2) will be satisfied automatically as none of the above three classes contains a 2--connected graph with either zero or one vertex.
And it only remains for us to check that conditions (C1) and (C3) are satisfied in all three cases.

\subsection{Cayley trees}\label{subsec:forests}

Our first application concerns one of the most basic subcritical graph class: forests of Cayley trees.
We note that the case of maximal independent sets was already discussed in \cite{MeirMoon88}.
In both maximal independent sets and maximal matchings, we proceed following the block-decomposition of trees, and we explicitly give the generating functions $B_0,\,B_1$ and $B_2$.
Notice that in a tree, blocks are reduced to single edges.
The computations of the constants given in Table 1 are obtained by computing the branch points of the corresponding system, using the explicit expressions for $B_0,\,B_1$ and $B_2$.

\subsubsection{Maximal independent sets in trees}

From the possible choices of an independent set in a single edge, namely
\begin{equation*}
	B(x,y_0,y_1,y_2) = \frac{x^2}2\left( 2 y_0y_1 + y_2^2 \right),
\end{equation*}
we obtain that
\begin{equation*}
    B_0 = xy_1,\,\, B_1 = xy_0,\,\, B_2 = xy_2.
\end{equation*}
All these functions are entire functions whose value at $x=0$ is equal to zero.
Also, for each choice of $y_0,y_1$ and $y_2$, the radius of convergence of $B(x,y_0,y_1,y_2)$ is infinite.
Thus, it is obvious that
\begin{equation*}
    \lim_{x\to\infty} \frac{\partial^2 B}{\partial y_2^2} = \lim_{x\to\infty} x = \infty.
\end{equation*}
So the extended subcriticality conditions are satisfied, and Lemma~\ref{lem:Le6} applies.

We finally have to check that the variance constant $\sigma_1^2$ is in fact positive (even if we do not compute it).
For this purpose, we apply the strategy of \cite[Lemma 4]{DrFuKaKrRu11} by doing formally an elimination procedure in order to reduce the system~\eqref{eqLe4syst} of three (positive) equations into one (positive) equation
of the form $C_0 = F(x,C_0,y_0)$ with solution $C_0 = C_0(x,y_0,1)$, where $F$ is a proper generating function in three variables.
It is easy to check that the conditions of~\cite[Lemma 4]{DrFuKaKrRu11} are satisfied.
Hence, $\sigma_1^2 > 0$.

\subsubsection{Maximal matchings in trees}

Observe that in this case $\bar B(x,z_0,z_1,z_2) = \frac{x^2}2\left( 2 z_0z_2 + z_1^2 + z_2^2 \right)$, which gives:
\begin{equation*}
    \bar B_0 = xz_2,\,\,\bar B_1 = xz_1,\,\,\bar B_2 = x(z_0 + z_2).
\end{equation*}
Hence, we are in a similar situation as above and Lemma~\ref{lem:Le6} applies.
This completes the proof for maximal matchings in trees.
Note that~\cite[Lemma 4]{DrFuKaKrRu11} again applies (after a formal substitution procedure), so that $\sigma_2^2 > 0$.

\paragraph{Observation.}

Let $S$ be a subset of $\mathbb{N}$ containing $1$.
By considering the operator $\sum_{s\in S}x^s/s!$ instead of the general set operator $\exp(x)$, one can directly adapt Lemmas~\ref{lem:blocks_mis} and~\ref{lem:blocks_mm} to analyse maximal independent sets and maximal matchings in forests whose vertex degrees are restricted to the set $S$.
One can then extend the proofs of Theorems \ref{thm:main1} and \ref{thm:main2} to this case.

\subsection{Cacti graphs}\label{subsec:cacti}

A graph is said to be a \emph{cactus graph} if it does not contain $K_4^{-}$ as a minor.
Observe that from the minor definition of this class, a graph is a cactus graph if and only if its connected components are cacti graphs as well.
There is a convenient equivalent definition of cacti graphs in terms of their block decomposition: a connected cactus graph is a connected graph in which every block is either an edge or a (simple) cycle.
In fact, in the univariate case the associated EGF is given by:
\begin{equation*}\label{eq:univ_blocks_mis}
    B(x) = \frac{x^2}2 + \frac{1}2 \log\left(\frac{1}{1 - x}\right) - \frac{x}2 - \frac{x^2}4,
\end{equation*}
where we deleted the cycles of size one and two, and we added the single edge.
In Subsections~~\ref{ss.MIS-cacti} and~\ref{ss.matching-cacti}, we use this characterisation to obtain explicit expressions for cycles refined with an independent set and a matching, respectively.

\subsubsection{Maximal independent sets in cacti graphs}\label{ss.MIS-cacti}

In this subsection, we consider the generating function counting cacti graphs carrying a maximal independent set.
According to our decomposition, we will first study blocks then prove Conditions (C1) -- (C3) on their associated generating function.

\paragraph{Unrooted blocks.} A cycle $C$ with an independent set $I$ can be decomposed in the following way: consider the set of vertices which are at distance at least two from $I$.
Two consecutive vertices $v$ and $w$ in this set are either adjacent or there exists a non-empty path alternating between vertices in $I$ and vertices at distance exactly one from $I$.
Conversely, such pair $C$ and $I$ can be obtained by taking a (possibly of size zero, one or two) base--cycle and by eventually replacing each edge by a sequence of paths whose vertices alternate between those in $I$ and those at distance one from $I$.
The vertices of the base--cycle will then become the vertices of $C$ at distance at least two from $I$.
We refer the readers to Figure~\ref{fig:cacti_mis} for an example of this decomposition.
\begin{figure}[H]
    \centering
    \includegraphics[scale=1]{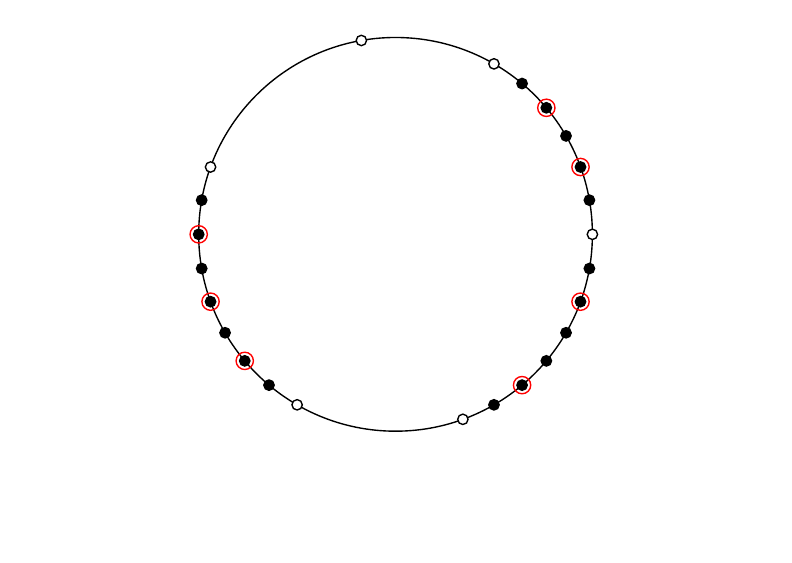}
    \includegraphics[scale=1]{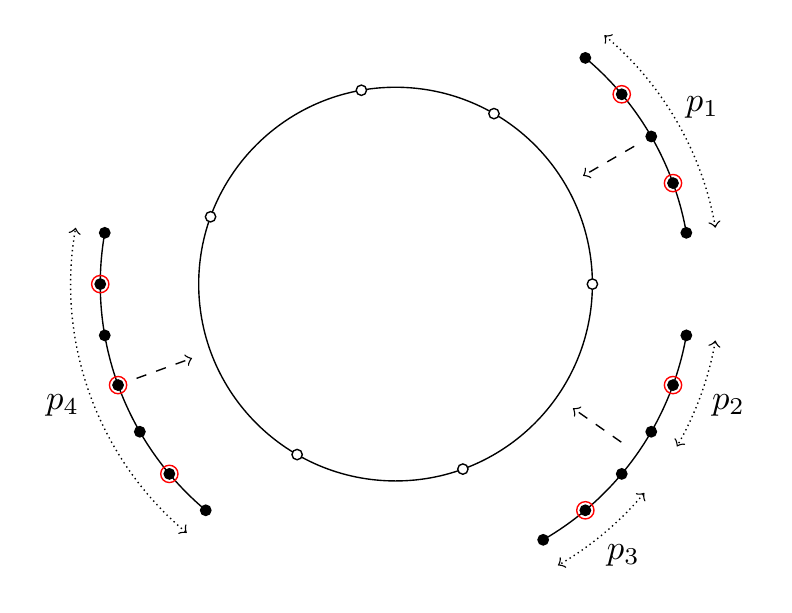}
    \caption{
        An independent set $I$ (vertices circled in red) on a cycle.
        Black vertices are at distance one from $I$.
        White vertices are at distance at least two from $I$.
        The rightmost picture shows the decomposition of a base-cycle carrying an independent set, using alternating paths.
    }
    \label{fig:cacti_mis}
\end{figure}
In this context, we say that a path is {\em alternating} if it is non-empty, if its terminals are vertices at distance one from the independent set, and if its vertices alternate between vertices in the independent set and vertices at distance one from the independent set.
The generating function of alternating paths is given by:
\begin{equation*}
    A(x,y_0,y_1)\equiv A = \frac{x^3y_0y_1^2}{1 - x^2y_0y_1}.
\end{equation*}

We now derive the generating function $U\equiv U(x,y_0,y_1,y_2)$ associated to the family of (unrooted) cycles carrying an independent set.
To that end, it is convenient to partition it into the subfamily of cycles carrying a maximal matching, whose associated GF will be denoted by $U_0\equiv U_0(x,y_0,y_1,y_2)$, and those with at least one vertex at distance two or more from the independent set, with associated GF $U_1\equiv U_1(x,y_0,y_1,y_2)$.
So that $U = U_0 + U_1$.
As discussed above, an equation for $U_1$ can be obtained by considering the base--cycle (whose vertices are encoded by the variable $y_2$) and substituting each edge by a possibly empty sequence of alternating paths:
\begin{equation*}\label{eq:U1_caci_mis}
    U_1 = \frac{1}{2}\log\left(\frac{1}{1 - xy_2(1 - A)^{-1}}\right) - \frac{xy_2}{2} - \frac{x^2y_2^2}{4},
\end{equation*}
where we removed the terms encoding base-cycles with only one or two vertices and empty sequences of alternating paths on the edges.
Similarly for $U_0$, we get:
\begin{equation*}\label{eq:U0_caci_mis}
    U_0 = \frac{1}{2}\log\left(\frac{1}{1 - A}\right) + \frac{1}{2}\log\left(\frac{1}{1 - x^2y_0y_1}\right) - \frac{x^2y_0y_1}{2},
\end{equation*}
where the first (resp. second) logarithm encodes cycles with an odd (resp. even) number of vertices.

The generating function of unrooted blocks is finally obtained from that of unrooted cycles plus the single edge:
\begin{equation*}\label{eq:B_cacti_mis}
    \begin{array}{ll}
        B\equiv B(x,y_0,y_1,y_2)
        & = U(x,y_0,y_1,y_2) + \ds\frac{x^2y_2^2}2 + x^2y_0y_1 \\
        & = \ds\frac{1}{2}\log\left( \frac{1}{1 - y_0y_1(y_1 - y_2)x^3 - x^2y_0y_1 - xy_2} \right) - \ds\frac{xy_2}2 + \ds\frac{x^2y_2^2}4 + \ds\frac{x^2y_0y_1}2.
    \end{array}
\end{equation*}

\paragraph{Singularity analysis.}

We now prove that $B$ verifies Conditions (C1) -- (C3) with $i=0$.
Observe first in Equation \eqref{eq:B_cacti_mis}, that the singular curve of $B$ is given by the roots of
\begin{equation*}
    x^3y_0y_1(y_1 - y_2) + x^2y_0y_1 + xy_2 - 1 = 0,
\end{equation*}
when $y_0$, $y_1$ and $y_2$ are positive real numbers.
We denote this curve by $R(y_0,y_1,y_2)$.
Notice that it satisfies (C1).
If we now consider the second derivative of $B$ with respect to $y_0$, which is a rational function:
\begin{equation*}
    \frac{\partial^2 B}{\partial y_0^2}(x,y_0,y_1,y_2) = \frac{(xy_1 - xy_2 + 1)^2y_1^2x^4}{2(1 - y_0y_1(y_1 - y_2)x^3 - x^2y_0y_1 - xy_2)^2}.
\end{equation*}
Then it is direct to check that $\partial^2 B/\partial y_0^2(0,y_0,y_1,y_2) = 0/2 = 0$, and thus that (C2) is satisfied.
And that (C3) also holds by definition of $R(y_0,y_1,y_2)$, as $\partial^2 B/\partial y_0^2 \to \infty$ when $x\to R(y_0,y_1,y_2)$.

Let us finally mention that $\sigma_1^2 > 0$ holds in this case, and that it can be shown in the same way as for Cayley trees.

\subsubsection{Maximal matchings in cacti graphs}\label{ss.matching-cacti}

In order to study the generating function of cacti graphs with a maximal matching, we first compute the generating function $\bar B\equiv \bar B(x,z_0,z_1,z_2)$ counting unrooted matched blocks.
To that end, we first obtain an equation satisfied by the generating function $V(x,z_0,z_1,z_2)$ counting unrooted cycles carrying a matching.

In what follows, a path carrying a maximal matching will be called a {\em matching path}.
Notice that such a path may contain vertices not incident with an edge of the matching.
Those vertices will later play the role of the vertices in the independent set of the matched blocks.
Observe then that any unrooted cycle carrying a matching can be obtained by eventually replacing each edge of a base--cycle (of size possibly zero, one or two) by a matching path.
The marginal vertices of the resulting cycle will then correspond to the original vertices of the base--cycle.
See Figure \ref{fig:cacti_match} for an illustration of this decomposition.
As they can be described as a sequence of edges, it is rather direct to obtain the generating function counting matching paths:
\begin{equation*}
    P(x,z_0,z_1) = xz_0 + (1 + xz_0)\frac{x^2z_1^2(1 + xz_0)}{1 - x^2z_1^2(1 + xz_0)}.
\end{equation*}
\begin{figure}[H]
    \centering
    \includegraphics[scale=1]{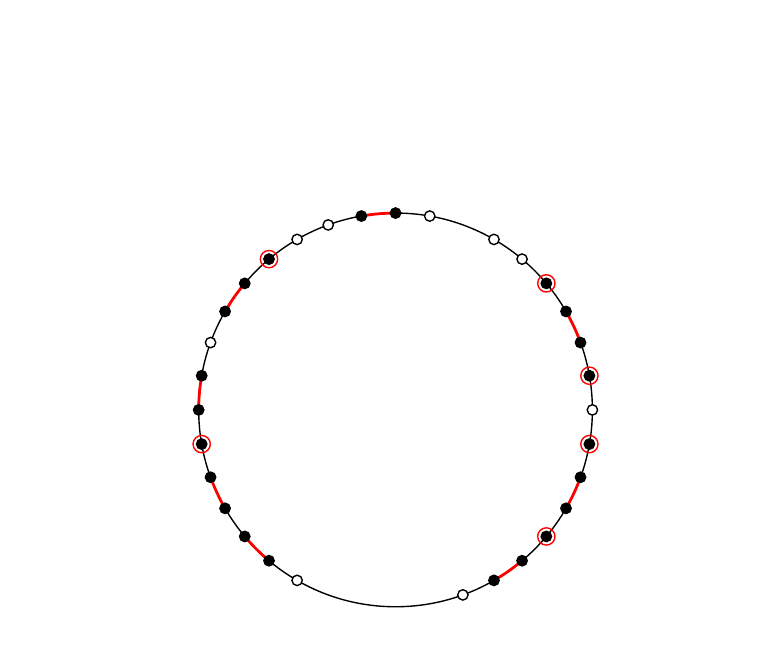}
    \includegraphics[scale=1]{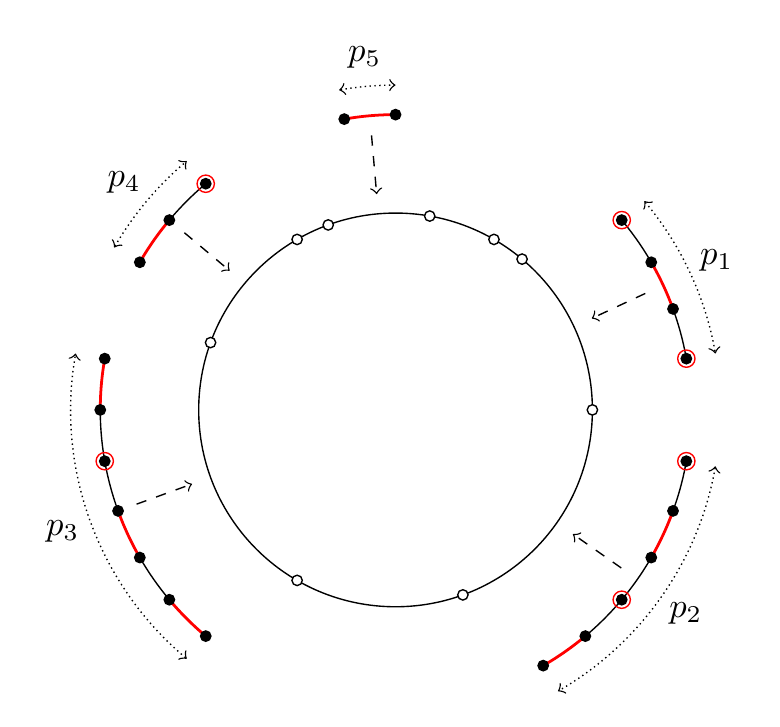}
    \caption{
        On the left, a matched block $(I,M,b)$.
        Vertices in red--black are the vertices in $I$.
        Red edges are edges in $M$.
        On the right, the decomposition of the matched block with matching paths.
    }
    \label{fig:cacti_match}
\end{figure}

From there and similarly to the case of independent sets, we partition the unrooted cycles in two sub-families, whether they have zero or at least one marginal vertex.
The associated generating functions will respectively be denoted by $V_0(x,z_0,z_1)$ and $V_1(x,z_0,z_1,z_2)$.
For instance, $V_0(x,z_0,z_1)$ counts unrooted cycles with a maximal matching.
Those cycles can be directly obtained from a matching path of size at least three, with at least one endpoint incident to the matching, and whose two ends are joined by a non matched edge so as to create a cycle.
Thus we get:
\begin{align*}
	V_0(x,z_0,z_1) & = \frac{1}{2}\log\left( \frac{1}{1 - x^2z_1^2(1 + xz_0)} \right) - \frac{x^2z_1^2}2, \\
    V_1(x,z_0,z_1,z_2) & = \frac{1}{2}\log\left( \frac{1}{1 - xz_2(1 + P(x,z_0,z_1))} \right) - \frac{xz_2(1 + xz_0)}2 - \frac{x^2z_2^2}4.
\end{align*}
Both equations can be understood similarly to their counterparts for independent sets.
The main difference is with the discarded term $x^2z_2z_0/2$ appearing in the second equation.
It encodes the case where the base--cycle is a loop with a single marginal vertex, and a matching path reduced to a single vertex (encoded by the first term of $P(x,z_0,z_1)$, that is $xz_0$) pasted on its edge.

The generating function of unrooted blocks is then obtained from that of unrooted cycles plus the single edge:
\begin{equation}\label{eq:B_cacti_match}
    \begin{array}{ll}
        \bar B
        &= V(x,z_0,z_1,z_2) + x^2z_0z_1 + \ds\frac{x^2z_2^2}2 \\
        & = \frac{1}{2}\log\left( \ds\frac{1}{1 - x^3z_0z_1^2 - (z_0z_2 + z_1^2)x^2 - xz_2} \right) + x^2z_0z_1 + \ds\frac{x^2z_2^2}2 - x^2z_1^2 - \ds\frac{xz_2(1 + xz_0)}2.
    \end{array}
\end{equation}

\paragraph{Singularity analysis.}

Notice in the equation above, that the singular curve of $\bar B$ is given by the roots $1 - x^3z_0z_1^2 - (z_0z_2 + z_1^2)x^2 - xz_2 = 0$, when $z_0$, $z_1$ and $z_2$ are positive real numbers.
As we did in the case of maximal independent sets, we can analogously show that Conditions (C1) -- (C3) are satisfied.
And again, it follows that $\sigma_2^2 > 0$.

\subsection{Series--parallel graphs}\label{subsec:sp}

Recall that a graph is said to be \emph{series--parallel} (SP for short) if it does not contain the graph $K_4$ as a minor.
As $K_4$ is a 3--connected graph, it follows that a graph is SP if and only if its connected and 2--connected components are SP as well.
So, similarly to the study of cacti graphs, our objective is to obtain expressions for 2--connected SP graphs with either an independent set or a matching.

This is a more complex family than cacti graphs.
We briefly recall the strategy to get enumerative formulas for the EGF of 2--connected SP graphs.
All the details can be found in \cite{BGKN07}.

We first study the EGF $B(x,y)$, where $y$ marks the number of edges.
To get an expression for $B(x,y)$, one needs to introduce an auxiliary graph class: a {\em series--parallel network} is defined as a simple graph with two distinguished vertices, called the \emph{poles} of the network (which are not labelled), and such that adding an edge between the two poles creates a 2--connected multigraph.
If there is an edge joining the two poles, it is called the \emph{root edge} of the network.
We denote by $D(x,y)$ the EGF of SP--networks.
In this situation $D(x,y)$ uniquely determines $B(x,y)$, as shown by the next equation:
\begin{equation*}
    2(1 + y)\frac{\partial B(x,y)}{\partial y} = x^2 + x^2D(x,y),
\end{equation*}
which can be understood as follows.
Any network can be obtained by marking and orienting an edge of a $2$--connected graph $G$, and adding back the labels of its two endpoints: the poles.
The summand $x^2$ encodes the single rooted edge that cannot be obtained that way.

We now decompose networks.
A network is said to be \emph{series} when it is obtained from a cycle with a distinguished edge, whose endpoints become the poles, and such that every other edge is replaced by a non--series network.
A network is said to be \emph{parallel} when it is obtained by gluing two or more non-parallel networks, none of them containing the root edge, along their poles.
In the situation of SP--graphs, those are all the possible types of networks, at the exception of the trivial network consisting of a single rooted edge.
We can then translate this decomposition into generating functions.
Denoting by $S(x,y)$ and $P(x,y)$ the EGF for series and parallel networks, we have the following relations (see \cite{WALSH198212} and also \cite{tutte1966}):
\begin{eqnarray*}
    D(x,y) & = & y + S(x,y) + P(x,y), \\
    S(x,y) & = & xD(x,y)(D(x,y) - S(x,y)), \\
    P(x,y) & = & y\exp_{\geq 1}(S(x,y)) + \exp_{\geq 2}(S(x,y)).
\end{eqnarray*}

It will be useful to use the following relation that can be obtained by combining the previous equations:
\begin{equation}\label{eq:modification}
    \frac{\partial B}{\partial y} = \frac{x^2}2 \exp(S(x,y)).
\end{equation}
The combinatorial property behind this relation is that an edge--rooted series--parallel graph (that corresponds to the generating function $\partial B/\partial y$) can be seen as a series--parallel network between the two vertices of the root--edge, consisting of this edge and a collection of series-networks between the two vertices.

In the next two subsections, we will refine the above decomposition in order to encode maximal independent sets and maximal matchings, respectively.

\subsubsection{Maximal independent sets in series--parallel graphs.}
We are now concerned with the generating functions of SP graphs carrying a maximal independent set.
As above, the vertices of the graphs carrying an independent set $I$ are said to be of \emph{type $i$} ($i\in \{0,1\}$), when they are at distance $i$ from $I$, and of \emph{type 2} otherwise.

\paragraph{Series--parallel networks.}

We denote by $D_{ij}$ the EGF of SP networks whose poles are of type $i$ and $j$, respectively.
Observe that by symmetry $D_{ij} = D_{ji}$, so we can restrict the range of the pairs of indexes $ij$ to the set $\{00, 01, 02, 11, 12, 22\}$.
The network $D_{ij}$ is either the single rooted edge $e_{ij}$, where $e_{01} = e_{22} = y$ and $e_{ij} = 0$ otherwise, a series network counted by the generating function $S_{ij}$, or a parallel network counted by the generating function $P_{ij}$.
The decomposition for a particular series network is illustrated in Figure \ref{fig:series_decomp_mis}, and one for a parallel network in Figure \ref{fig:parallel_decomp_mis}.
\begin{figure}
    \centering
    \includegraphics[scale=1]{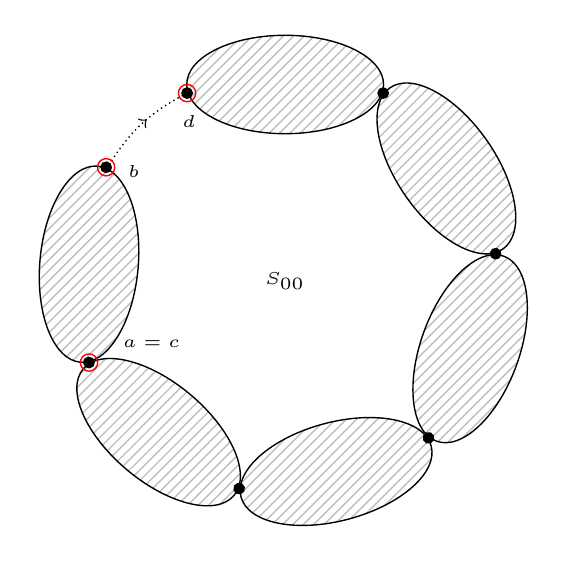}
    \qquad \qquad \qquad
    \includegraphics[scale=1]{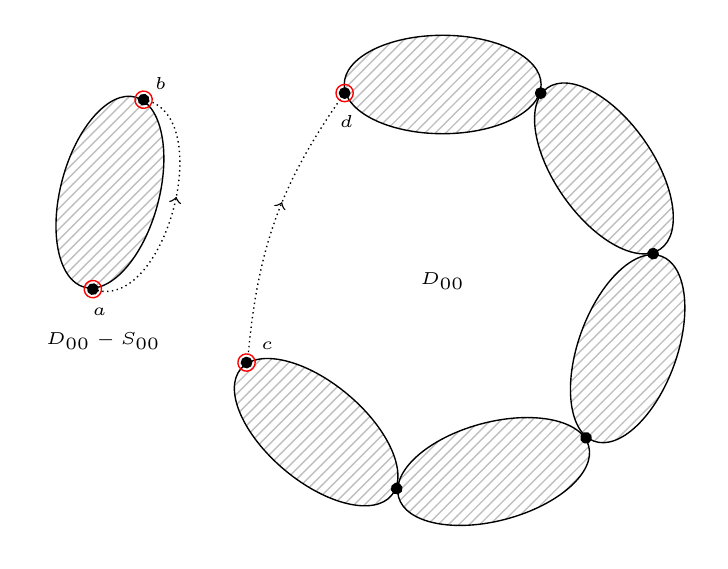}
    \caption{
        Decomposition of a series network carrying an independent set $I$ (in red) and counted by the generating function $S_{00}$.
        It is obtained by identifying poles $a$ and $c$ (that are both in $I$) of the two networks on the right.
        In the general case, $a$ (resp. $c$) could also be of type 1 and then $c$ (resp. $a$) of type 1 or 2, or they could both be of type 2.
    }
    \label{fig:series_decomp_mis}
\end{figure}
\begin{figure}
    \centering
    \includegraphics[scale=1]{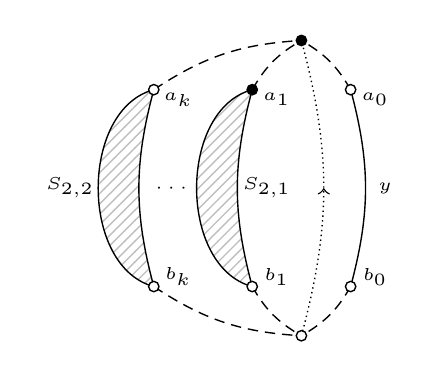}
    \caption{
        A typical parallel network, carrying an independent set, and counted by the generating function $P_{2,1}$.
        It is obtained by identifying together the upper poles of each non--parallel (including the single edge) networks $a_0, a_1, \ldots, a_k$ then the lower poles $b_0, b_1, \ldots, b_k$.
        As the resulting upper pole is of type 1, exactly one of the $a_i's$ ($i\in\{1,\ldots,k\}$) must be of type 1 (in black) while all the others must be of type 2 (in white).
        Similarly, all the $b_i's$ ($i\in\{0,\ldots,k\}$) must be of type 2.
    }
    \label{fig:parallel_decomp_mis}
\end{figure}
We can specify those generating functions via the following positive system of 18 equations and 18 unknowns:
\begin{align*}
    & D_{ij} = e_{ij} + S_{ij} + P_{ij}, \\
    & S_{ij} = D_{i0}xy_0(D_{0j}-S_{0j}) + (D_{i1} + D_{i2})xy_1(D_{1j}-S_{1j}) + (D_{i1}y_1 + D_{i2}y_2)x(D_{2j}-S_{2j}), \\
    & P_{00} = \exp_{\geq 2}(S_{00}), \\
    & P_{01} = y \exp_{\geq 1}(S_{01} + S_{02}) + \exp_{\geq 2}(S_{01}) + \exp_{\geq 1}(S_{01}) \exp_{\geq 1}(S_{02}), \\
    & P_{02} = \exp_{\geq 2} (S_{02}), \\
    & P_{11} = \exp_{\geq 2}(S_{11}) + \exp_{\geq 1}(S_{11})(y\exp(2S_{12}+S_{22}) + \exp_{\geq 1}(2S_{12}+S_{22})) \\
        & \qquad + (1+y)\exp_{\geq 1}(S_{12})^2\exp(S_{22}), \\
    & P_{12} = y \exp_{\geq 1}(S_{12}) \exp(S_{22}) + \exp_{\geq 2}(S_{12}) + \exp_{\geq 1}(S_{12}) \exp_{\geq 1}(S_{22}), \\
    & P_{22} = y \exp_{\geq 1}(S_{22}) + \exp_{\geq 2}(S_{22}).
\end{align*}

In order to proceed further, we eliminate $D_{ij}$ from this system to obtain a positive and strongly connected system of equations for $S_{ij} = S_{ij}(x,y,y_0,y_1,y_2)$ and $P_{ij}= P_{ij}(x,y,y_0,y_1,y_2)$,
where the right hand--side consists of entire functions.
Note that for the equations defining $S_{ij}$, the term $D_{ij}-S_{ij}=e_{ij}+P_{ij}$ makes the whole system positive.
Thus, all functions have a common singular behaviour that is again of square--root type:
\begin{equation*}
    S_{ij}(x,y,y_0,y_1,y_2) = s_{0;ij}(y,y_0,y_1,y_2) +s_{1;ij}(y,y_0,y_1,y_2) \left({1- \frac x{\rho(y,y_0,y_1,y_2)}}\right)^{1/2} + \cdots,
\end{equation*}
and
\begin{equation*}
    P_{ij}(x,y,y_0,y_1,y_2) = p_{0;ij}(y,y_0,y_1,y_2) +p_{1;ij}(y,y_0,y_1,y_2) \left({1- \frac x{\rho(y,y_0,y_1,y_2)}}\right)^{1/2} + \cdots,
\end{equation*}
where $s_{1;ij}(y,y_0,y_1,y_2)<0$ and $p_{1;ij}(y,y_0,y_1,y_2)<0$ for positive $y,y_0,y_1,y_2$.

\paragraph{2--connected series--parallel graphs.}

The next step is to relate these network generating functions with the generating function $B\equiv B(x,y,y_0,y_1,y_2)$ of independent sets in 2--connected series parallel graphs.
We adapt now Equation \ref{eq:modification}, which in our present situation has the following expression:
\begin{align*}
    \frac{\partial B}{\partial y}
        & = x^2y_0y_1\exp(S_{01} + S_{02}) + \frac{x^2}2 y_2^2\exp(S_{22}) + x^2 y_1y_2\exp _{\geq 1}(S_{12})\exp (S_{22}) \\
        & + \frac{x^2}2 y_1^2\left(\exp(S_{11} + 2S_{12} + S_{22}) - 2\exp (S_{12} + S_{22}) + \exp(S_{22})\right).
\end{align*}
This is immediate by considering all possible situations for the rooted edge.
Observe then that the radius of convergence of $\partial B/\partial y$ coincides with the one of networks, so that (C1) is satisfied.
Notice also that Condition (C2) is trivially satisfied.

Let us finally argue on (C3).
Observe that despite the negative terms, $\partial B/\partial y$ is in fact a positive function of the generating functions $\{S_{ij}\}_{i,j=0,\dots,2}$.
Hence, $\partial B/\partial y$ also admits a square--root singularity:
\begin{equation*}
    \frac{\partial B}{\partial y} = b_{0}(y,y_0,y_1,y_2) + b_{1}(y,y_0,y_1,y_2)\left({1 - \frac x{R(y,y_0,y_1,y_2)}}\right)^{1/2} + \cdots,
\end{equation*}
where $b_{1}(y,y_0,y_1,y_2)<0$ for positive $y,y_0,y_1,y_2$.
Next, by applying the proof method of \cite[Lemma 2.28]{DrmotaBook},
we can integrate $\partial B/\partial y$ with respect to $y$, and then take the derivative with respect to $y_0$, to obtain the same kind of square--root singularity for $\partial B/\partial y_0$:
\begin{equation*}
    \frac{\partial B}{\partial y_0} = b_{1,0}(y,y_0,y_1,y_2) + b_{1,1}(y,y_0,y_1,y_2)\left({1 - \frac x{R(y,y_0,y_1,y_2)}}\right)^{1/2} + \cdots,
\end{equation*}
and consequently the following representation of $\partial^2 B/\partial y_0^2$:
\begin{equation*}
    \frac{\partial^2 B}{\partial y_0^2} = b_{2,-1}(y,y_0,y_1,y_2)\left({1 - \frac x{R(y,y_0,y_1,y_2)}}\right)^{-1/2} + b_{2,1}(y,y_0,y_1,y_2) + \cdots,
\end{equation*}
which implies that Condtion (C3) holds for $i = 0$.
This completes the proof for maximal independent sets in series--parallel graphs, given that $\sigma_1^2 > 0$, which is proven as in the previous cases.

\subsubsection{Maximal matchings in series--parallel graphs}

We proceed similarly to the case of independent sets.
Let $G$ be a series--parallel graph with a matching $M$ and an independent set $I$ such that $I\cap V(M)=\emptyset$.
A vertex $v$ of $G$ is said to be of \emph{type 0} when $v\in I$, of \emph{type 1} when $v\in V(M)$ and of \emph{type 2} otherwise.

\paragraph{Series--parallel networks.}

Let $\bar D_{ij}(x,y,z_0,z_1,z_2)$ be the exponential generating function counting matchings in series--parallel networks whose poles are of type $i$ and $j$.
As before, observe that $\bar D_{ij} = \bar D_{ji}$ and for $ij\in \{00,01,02,11,12,22\}$, define $\bar S_{ij}$ and $\bar P_{ij}$ to be the generating functions counting matchings in networks that are respectively series and parallel.

The following system of 18 equations and 18 unknowns holds:
\begin{align*}
    & \bar D_{ij} = e_{ij} + \bar S_{ij} + \bar P{ij}, \\
    & \bar S_{ij} = (\bar D_{i0} - \bar S_{i0})xz_0\bar D_{0j} + (\bar D_{i1} - \bar S_{i1})xz_1\bar D_{2j} + (\bar D_{i2} - \bar S_{i2})x(z_1\bar D_{1j} + z_2\bar D_{2j}), \\
    & \bar P_{00} = \exp_{\geq 2} (\bar S_{00}), \\
    & \bar P_{01} = \bar S_{01}(y\exp (\bar S_{02}) + \exp_{\geq 1}(\bar S_{02})), \\
    & \bar P_{02} = y\exp_{\geq 1}(\bar S_{02}) + \exp_{\geq 2}(\bar S_{02}), \\
    & \bar P_{11} = (y\bar S_{11} + (1+y)\bar S_{12}^2)\exp(\bar S_{22}) + (y + \bar S_{11})\exp_{\geq 1}(\bar S_{22}), \\
    & \bar P_{12} = \bar S_{12}(y\exp(\bar S_{22}) + \exp_{\geq 1}(\bar S_{22})), \\
    & \bar P_{22} = y\exp_{\geq 1}(\bar S_{22}) + \exp_{\geq 2}(\bar S_{22}),
\end{align*}
where this time $e_{02} = e_{11} = e_{22} = y$ and $e_{ij} = 0$.

\paragraph{2--connected series--parallel graphs.}

It remains to check the relevent analytic properties of $\bar B\equiv \bar B(x,y,z_0,z_1,z_2)$ in order to ensure that Lemma \ref{lem:Le6} can be applied.
By eliminating $\bar D_{ij}$ from the above system, we again get a positive and strongly connected system of equations for the set of generating functions $\{\bar S_{ij},\bar P_{ij}\}_{i,j=0,\dots,2}$, where the right hand--side consists of entire functions.
In particular, the functions $\bar S_{ij}$ and $\bar P_{ij}$ all have a common square--root singular behaviour.

And we have that:
\begin{align*}
    \frac{\partial \bar B}{\partial y}
        & = x^2z_0z_1\bar S_{01}\exp(\bar S_{02}) + x^2z_0z_2\exp(\bar S_{02}) + x^2z_1z_2\bar S_{12}\exp(\bar S_{22}) \\
        & + \frac{x^2}{2}z_2^2\exp(\bar S_{22}) + \frac{x^2}{2}z_1^2\left(\bar S_{11} + \bar S_{12}^2 + 1 \right)\exp(\bar S_{22}).
\end{align*}

Finally using the very same arguments as in the case of maximal independent sets, we show that (C1) -- (C3) are satisfied in the context of maximal matchings in series--parallel graphs.
This completes the proof.
As above, let us add that we can similarly show $\sigma_2^2 > 0$.

%
%

\section{Remarks and further research}\label{sec:remarks and further research}

In this paper, we have studied maximal independent sets and maximal matchings in subcritical graph classes.
In particular, we have provided a detailed analytic study for trees, cacti graphs and series--parallel graphs.
We would like to point that our techniques can be extended in several directions.
First, by a slight modification of our encoding one would be able to study independent sets and matchings that are not necessarily maximal.
Secondly, we could in principle analyse other graph classes in the context of subcritical families: this would include outerplanar graphs and graphs defined by a finite set of $3$--connected components.
In the latter, one would need to extend the decomposition grammar used to encode SP--networks to networks arising from $3$--connected components (see for instance \cite{GiNoRu13}).
Finally, a more refined (and longer) study would provide explicit constants for the asymptotic estimate of the variance.
Notably, this could be arranged with some work in the case of SP--graphs (the analysis and computation for the expectation is detailed in Appendix \ref{app:SP}).

Some problems that we have not been able to study are the following: in statistical physics language, our model can be understood as an \emph{annealed model}, where a combinatorial parameter $v$ on a class $\mathcal{C}$ is studied by averaging over all pairs ($g\in\mathcal{C}$, $v(g)$).
It would be very interesting to study both maximal independent sets and maximal matchings from the point of view of the \emph{quenched model}, in which one first averages on $\mathcal{C}$, for example by picking a graph $\gamma$ uniformly in $\mathcal{C}$, then studies the random variable $v(\gamma)$.
In this situation however, new ideas are needed as our encoding is not sufficient to encapsulate this information.

Finally, a challenging problem is to study both annealed and quenched models for planar graphs.
The first step towards the study of the annealed model would be to have access to 3--connected planar graphs carrying either an independent set or a matching, which already seems a testing problem in map enumeration.

\paragraph{Acknowledgements}

This work was started during a visit of the first author at Freie Uniersit\"at Berlin.
The authors would like to thank Mireille Bousquet--Mélou for explaining the connections of our models with the ones that arise in statistical physics, and the anonymous referees of both the final version of this paper and the extended abstract~\cite{DRRR18} of this work, which was presented at the conference \emph{Analysis of Algorithms'18}, for suggesting several improvements.

%
%

\bibliographystyle{abbrv} 	
\bibliography{mis-biblio}

%
%
\newpage
\appendix
\section{Appendix: Computations of the constants}\label{app:appendix}

\subsection{Computations for cacti graphs}\label{app:cacti}

\paragraph{Number of maximal independent sets.}

From Equation \eqref{eq:B_cacti_mis}, one can compute the GFs $B_i\equiv B_i(x,y_0,y_1,y_2)$:
\begin{align*}\label{eq:Bi_cacti_mis}
    & B_i = \frac{1}{x}\frac{\partial B}{\partial y_i}(x,y_0,y_1,y_2)
    & (i\in\{0,1,2\}).
\end{align*}
Notice that $B_0$, $B_1$ and $B_2$ are all in fact rational functions of $x$, $y_0$, $y_1$ and $y_2$.
Substituting now each $B_i$ ($i\in\{0,1,2\}$) by its rational form in the system of equations of Lemma \ref{lem:blocks_mis}, we obtain a system of three implicit equations in variables $C_0(x,y_0,y_1)$, $C_1(x,y_0,y_1)$, $C_2(x,y_0,y_2)$, $x$, $y_0$ and $y_1$ only.

Let us now denote the above system by ${\cal S}$ and the determinant of its Jacobian by ${\cal DJ(S)}$.
After setting $y_0=y_1=1$ and $x=\rho_1$, one can obtain, using for example {\tt Maple/fsolve}, numerical values up to any degree of precision of the four solutions of the system ${\cal S \cup DJ(S)}$:
\begin{equation}\label{eq:values_cacti_mis}
    \begin{array}{lllll}
        & \rho_1 \approx 0.1867604863, &&
        & C_0(\rho_1) \approx 1.5865482480, \\
        & C_1(\rho_1) \approx 0.6912173223, &&
        & C_2(\rho_1) \approx 1.1812890598.
    \end{array}
\end{equation}

Finally, recall that in Section \ref{subsec:cacti} we already proved that
\begin{equation*}
    |{\cal I}_n| \sim A_1\cdot n^{-5/2}\cdot \rho_1^{-n}\cdot n!,
\end{equation*}
where $A_1>0$, and $\rho_1$ is as above.
So that together with the following estimate (see \cite[Theorem 15]{DrFuKaKrRu11}) of the number of cacti graphs on $n$ vertices:
\begin{align*}
    & g_n \sim g\cdot n^{-5/2}\cdot \rho^{-n}\cdot n!,
    & \text{where } \rho \approx 0.2387401437 \text{ and } g>0,
\end{align*}
one can obtain an estimate of the number of maximal independent sets in a uniformly at random cactus graph on $n$ vertices:
\begin{align*}
    & AI_n = \frac{|\mathcal{I}_n|}{g_n} \sim C \cdot \alpha^{n},
    & \text{where } C = A_1/g > 0  \text{ and } \alpha = \rho/\rho_1 \approx 1.2783225640.
\end{align*}

\paragraph{Expected size of a maximal independent set.}

In the same setting, we can now estimate the expected size of a typical maximal independent set.
It works as follows.
Observe that the variable $y_0$ in any of the generating functions $C_i$ ($i=0,1,2$) encodes the size of a maximal independent set.
So that after setting $x = \rho_1(y_0)$ in the system ${\cal S \cup DJ(S)}$, we can differentiate it to obtain a system characterising $\rho_1'(y_0)$ (see the proof of \cite[Theorem 2.35]{DrmotaBook}).
From there, we can again set $y_0=y_1=1$ to substitute $\rho_1(1) = \rho_1$ and $C_i(\rho_1)$ ($i=0,1,2$) with their approximated values in \eqref{eq:values_cacti_mis}.
This gives us a system characterising uniquely $\rho_1'(1)$ that we can again solve numerically.
In particular, we get that:
\begin{equation*}
    \rho_1' = \rho_1'(1) \approx -0.0805687207.
\end{equation*}

And we conclude by \cite[Theorem 2.35]{DrmotaBook}:
\begin{align*}
    & \mathbb{E}[SI_n] = \mu n + O(1),
    & \text{where } \mu = -\frac{\rho_1}{\rho_1'} \approx 0.4314013220.
\end{align*}

\paragraph{Number of maximal matchings.}

From Equation \eqref{eq:B_cacti_match}, one can now compute the rational GFs $\bar B_i\equiv \bar B_i(x,z_0,z_1,z_2)$:
\begin{align}\label{eq:Bi_cacti_match}
    & \bar B_i = \frac{1}{x}\frac{\partial \bar B}{\partial z_i}(x,z_0,z_1,z_2)
    & (i\in\{0,1,2\}).
\end{align}
As above, we substitute each $\bar B_i$ in the system given in Lemma \ref{lem:blocks_mm} by their rational forms in \eqref{eq:Bi_cacti_match}.
Now after computing the Jacobian of the resulting system, one can obtain the approximated values of each variable when setting $z_0=z_1=1$, as in the case of independent sets:
\begin{equation*}
    \begin{array}{lllll}
        & \rho_2 \approx 0.2016232044, &&
        & \bar C_0(\rho_2) \approx 1.4072706847, \\
        & \bar C_1(\rho_2) \approx 0.6446283865, &&
        & \bar C_2(\rho_2) \approx 1.3374106486.
    \end{array}
\end{equation*}
This implies that the number of maximal matchings in a uniformly at random cactus graph on $n$ vertices is asymptotically:
\begin{align*}
    & AM_n \sim D \cdot \beta^{n},
    & \text{where } D > 0 \text{ and } \beta \approx 1.1840906130.
\end{align*}

\paragraph{Expected size of a maximal matching.}

We proceed in a similar manner as for the expected size of a maximal independent set.
Notice that this time however, it is the variable $z_1^2$ which encodes the size of the maximal matchings.
And we get:
\begin{equation*}
    \rho_2' = \rho_2'(1) \approx -0.6934685099.
\end{equation*}
So that:
\begin{align*}
    & \mathbb{E}[SM_n] = \lambda n + O(1),
    & \text{where } \lambda = -\frac{1}{2}\frac{\rho_2}{\rho_2'} \approx 0.3467342549.
\end{align*}

\subsection{Computations for series--parallel graphs}\label{app:sp-graphs}

In this subsection, we obtain some of the constants for the case of series--parallel graphs.
This situation is slightly different compared to that of cacti graphs: the system of equations that we obtain is too complex, so that we cannot obtain direct asymptotic estimates.
In particular, we did not find a way to directly compute the constant growth for connected series--parallel graphs.
Our strategy will then consist in adapting the grammar introduced in~\cite{CFKS08} in order to obtain equations directly from networks.
As shown in the following, this methodology provides an even longer system of equations but that can be treated by symbolic programs efficiently.

In the rest of this Appendix and for each graph classes, we will use $\rho_1$, $\rho_2$ for the radius of convergence in the study of maximal independent sets and maximal matchings, respectively.

\subsubsection{Rooted dissymmetry: from edge--rooted to vertex-rooted graphs}\label{app:SP}

The generating function counting vertex--rooted 2--connected SP graphs from edge--rooted ones can be obtained via a rooted analogue of the so-called {\sl dissymmetry theorem for trees}, as designed for 2--connected graphs in~\cite[Section 5.3.3]{CFKS08}.

\paragraph{\sl Tree decomposition of a 2--connected series--parallel graph.}

The recursive decomposition of a given 2--connected series--parallel graph $\gamma$ in terms of networks induces a unique tree $\tau(\gamma)$ (see \cite[Section 4.3]{CFKS08}), which is called the {\sl RM--tree} associated to $\gamma$.
The letters R and M represent the types of nodes of $\tau(\gamma)$: R--nodes or {\sl ring nodes} when the underlying network--decomposition is series, and M--nodes or {\sl multiple edge nodes} when the underlaying network--decomposition is parallel.

\paragraph{\sl Restricted $RM$--trees.}

If our 2--connected series--parallel graph $\gamma$ is rooted at a vertex $v$, it is useful to consider a {\sl restriction} of its associated RM--tree $\tau(\gamma)$: that is the subtree of $\tau(\gamma)$ induced by the nodes containing $v$ (see~\cite[Section 4.4]{CFKS08}).
Let then $\mathcal{T}$ be the family of all such restricted RM--trees.
The decomposition directly implies that there is a bijection between $\mathcal{T}$ and the family of 2--connected series--parallel graphs rooted at a vertex (enriched by either an independent set or a matching).
So that to study the latter family by means of generating functions, one can instead enumerate the family of restricted RM--trees.

\paragraph{\sl The dissymmetry theorem for trees.}

Now from $\mathcal{T}$, it is then comparatively easy to enumerate the different related families $\mathcal{T}^{\bullet}$, $\mathcal{T}^{\bullet-\bullet}$ and $\mathcal{T}^{\bullet\to\bullet}$ of the RM--trees rooted respectively at a node, at an edge or at a directed edge.
The main interest of this method is that one can exploit the following relation between those combinatorial classes, called the {\sl dissymmetry theorem for trees} (see~\cite[Section 4.1]{BerLaLe98} for a proof):
    \begin{equation*}
        \mathcal{T} \cup \mathcal{T}^{\bullet\to\bullet} \simeq \mathcal{T}^{\bullet} \cup \mathcal{T}^{\bullet-\bullet},
    \end{equation*}
which translates into a functional equation between the associated generating functions.
This ultimately allows us to access, via the bijection, the generating function counting 2--connected SP graphs rooted at a vertex.

Similarly to what was done for networks, one can extend such a decomposition to 2--connected SP graphs carrying an independent set or a matching.
This will be done in the next two subsections.

\subsubsection{Independent sets in 2--connected vertex-rooted series--parallel graphs}

As discussed above, one can associate a 2--connected SP graph (not reduced to a single edge), carrying an independent set $I$, and rooted a vertex $v$ at distance $i$ ($i=0,1,2$) from $I$, to a unique restricted RM--tree whose every node contains $v$.
For each $i\in\{0,1,2\}$, let then $T_i\equiv T_i(x,y_0,y_1,y_2)$ be the generating function counting those restricted RM--trees and observe that the bijection implies the equality $B_i - xy_i = T_i$.

Now in order to use the dissymmetry theorem on the $T_i$'s, we study next for each $i\in\{0,1,2\}$ the generating functions $T_i^M\equiv T_i^M(x,y_0,y_1,y_2)$, $T_i^R\equiv T_i^R(x,y_0,y_1,y_2)$, $T_i^{M\to R}\equiv T_i^{M\to R}(x,y_0,y_1,y_2)$ and $T_i^{M-R}\equiv T_i^{M-R}(x,y_0,y_1,y_2)$ respectively counting the restricted RM--trees rooted at an M-node, at an R--node, and at an edge (both directed and not) between an M--node and an R--node.
Notice first that due to the maximality of both the series and the parallel decompositions, there cannot be an edge between two R--nodes or two M-nodes in a restricted RM--tree.
And observe that by symmetry $T_i^{M\to R} = 2T_i^{M-R}$, for all $i\in\{0,1,2\}$.

\paragraph{\sl Restricted RM--tree rooted at an M--node.}

In this case, the underlaying graph carrying an independent set is decomposed from a parallel edge (with at least three edges), rooted at one of its endpoints, and in which every edge but at most one is substituted by a series network.
As before, we need to specify each generating function with respect to the type of the endpoints.
Let $M_{ij}\equiv M_{ij}(x,y_0,y_1,y_2,y)$ be such a generating function, counting the case where the two endpoints are of respective type $i$ and $j$.
The variable $y$ encodes the number of edges.
Again by symmetry it holds that $M_{ij} = M_{ji}$.
The {\sl textbook case} is the one where the two endpoints are of type 2:
\begin{equation*}
    M_{22} = y\exp_{\geq 2}(S_{22}) + \exp_{\geq 3}(S_{22}).
\end{equation*}
Remark now that in the cases where one of the endpoints is of type 0 and the other is of type 0 or 2, then all the edges of the parallel edge must be substituted by a series network:
\begin{equation*}
    M_{00} = \exp_{\geq 3}(S_{00}) \text{ and } M_{02} = \exp_{\geq 3}(S_{02}).
\end{equation*}
The decompositions of the remaining cases are slightly more involved, but can be proved following an argumentation in the same lines:
\begin{align*}
    & M_{01} = y\exp_{\geq 2}(S_{01} + S_{02}) + \exp_{\geq 2}(S_{02})S_{01} + \exp_{\geq 1}(S_{02})S_{01}^2/2 + \exp_{\geq 3}(S_{01})\exp(S_{02}) \\
    & M_{11} = \exp(2S_{12} + S_{22})(y\exp_{\geq 2}(S_{11}) + \exp_{\geq 3}(S_{11})) + \exp_{\geq 1}(2S_{12} + S_{22})(yS_{11} + S_{11}^2/2) \\
    & + \exp_{\geq 2}(2S_{12} + S_{22})S_{11} + (y\exp(S_{22}) + \exp_{\geq 1}(S_{22}))S_{12}^2/2 + (1 + y)\exp_{\geq 3}(S_{12}) \exp(S_{22}) \\
    & M_{12} = \exp_{\geq 3}(S_{12})\exp(S_{22}) + S_{12} \exp_{\geq 2}(S_{22}) + \exp_{\geq 1}(S_{22})S_{12}^2/2 \\
    & + y\left(\exp_{\geq 1}(S_{12})\exp_{\geq 1}(S_{22}) + \exp_{\geq 2}(S_{12})\right).
\end{align*}
And we can finally write the different generating functions counting the restricted RM--trees rooted at an M--node:
\begin{align*}
    & T_i^M = x\left(y_0M_{i0}+y_1M_{i1}+y_2M_{i2}\right)
    & (i\in\{0,1,2\}).
\end{align*}

\begin{figure}
    \centering
    \includegraphics[scale=1]{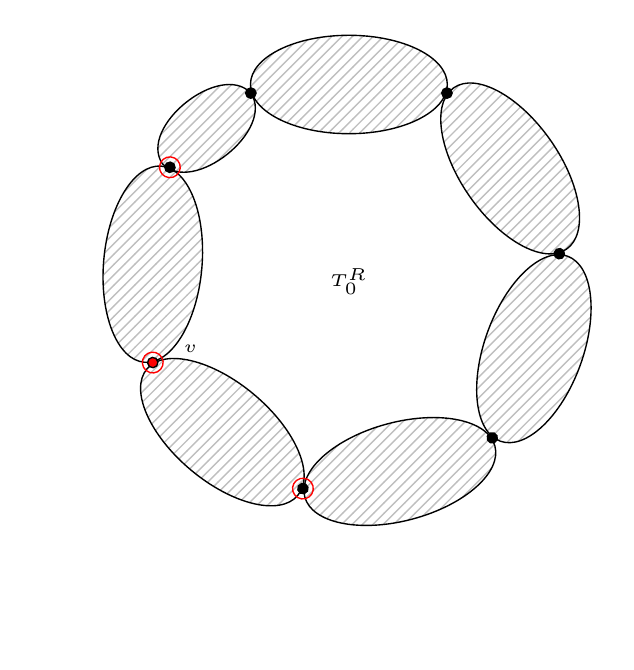}
    \qquad \qquad \qquad
    \includegraphics[scale=1]{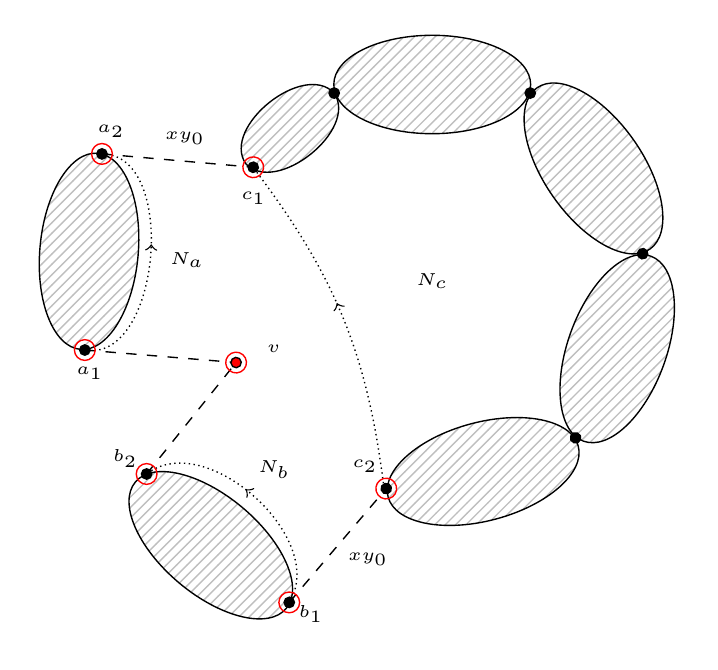}
    \caption{
        Decomposition of a RM-tree, restricted at a vertex $v$ of type 0, and rooted at an R-node.
    }
    \label{fig:Rnode_mis}
\end{figure}

\paragraph{\sl Restricted RM--tree rooted at an R--node.}

A typical case is illustrated in Figure~\ref{fig:Rnode_mis}: on each side of $v$ are two non--series networks $N_a$ and $N_b$, with respective poles $a_1, a_2$ and $b_1, b_2$, each of type 0 (thus, each encoded by $D_{00} - S_{00}$), where vertices $a_2$ and $b_1$ are identified with $v$, and vertices $a_1$ and $b_2$ are each identified with the poles $c_1, c_2$ of a third network $N_c$ (encoded by $D_{00}$).
In the latter, the vertices involved stop being poles and we must add their two labels: $x^2y_0^2$.
This gives $x^2(D_{00} - S_{00})^2y_0^2D_{00}/2$, where the factor 1/2 encodes the symmetry around $v$ when exchanging $N_a$ with $N_b$ and the two the poles of $N_c$.

The general decomposition for $i\in\{0,2\}$ is obtained by doing a careful (but similar to what was done above) analysis of each possible cases, and we have:
\begin{equation*}
    T_i^R = \frac{x^2}{2}\left( (D_{i0} - S_{i0})y_0A_{0i} + (D_{i1} - S_{i1})y_1(A_{1i} + A_{2i}) + (D_{i2} - S_{i2})(y_1A_{1i} + y_2A_{2i}) \right),
\end{equation*}
where for $ji\in\{00,01,02,10,11,12,20,21,22\}$:
\begin{equation*}
    A_{ji} = D_{j0}y_0(D_{0i} - S_{0i}) + (D_{j1} + D_{j2})y_1(D_{1i} - S_{1i}) + (D_{j1}y_1 + D_{j2}y_2)(D_{2i} - S_{2i}).
\end{equation*}
The latter notation is understood in Figure \ref{fig:Rnode_mis} as follows: $N_b$ is encoded by $D_{i0} - S_{i0}$ while $N_c$ and $N_a$ are encoded by $A_{0i}$ (the summands of $A_{0i}$ correspond to the possible types of the vertices $a_2$ and $c_1$, while those of $T_i^R$ correspond to the types of $b_1$ and $c_2$).
Using the same notation, we deal similarly with the (more involved) case $i=1$:
\begin{align*}
        & T_1^R = \ds\frac{x^2}2( (D_{10} - S_{10})y_0(A_{01} + A_{02}) + (D_{20} - S_{20})y_0A_{01} + (D_{22} - S_{22})(y_2A_{21} + y_1A_{11}) \\
        & \qquad\qquad + (D_{12} - S_{12})(2y_1(A_{11} + A_{12}) + y_2(A_{21} + A_{22})) \\
        & \qquad\qquad + (D_{11} - S_{11})y_1(A_{11} + A_{21} + A_{12} + A_{22}) ).
\end{align*}

\paragraph{\sl Restricted RM--tree rooted at an edge.}

As mentioned in the previous subsection, there can be no edge between two R--nodes or two M--nodes.
An edge between an R--node and an M--node of the restricted RM--tree happens when the poles of a series network are identified with the poles of a parallel network and such that the root vertex inducing the restricted tree is one of the poles.
In the three following equations, we specify this interaction in each of the cases depending on the types of the poles of both the series and the parallel network:
\begin{align*}
        & T_i^{R-M} = x(y_0S_{i0}P_{i0} + y_1(S_{i1}(P_{i1} + P_{i2}) + S_{i2}P_{i1}) + y_2S_{i2}P_{i2})
        & (i\in\{0,2\}), \\
        & T_1^{R-M} = x(y_0(S_{10}(P_{10} + P_{20}) + S_{20}P_{10}) + y_2(S_{12}(P_{12} + P_{22}) + S_{22}P_{12}) \\
        & \qquad\qquad + y_1(S_{11}(P_{11} + P_{12} + P_{21} + P_{22}) + 2S_{12}(P_{11} + P_{21}) + S_{22}P_{11})).
\end{align*}

\paragraph{\sl Vertex-rooted 2--connected SP graphs carrying an indepedent set.}

And we can finally apply the dissymmetry theorem for trees on each of the families of restricted RM-trees:
\begin{align}\label{eq:Bi_sp_mis}
    & B_i - xyy_i = T_i = T_i^M + T_i^R - T_i^{R-M}
    & (i\in\{0,1,2\}).
\end{align}

\subsubsection{Number and expected size of maximal independent sets}

As in the case of cacti graphs, we use \eqref{eq:Bi_sp_mis} to apply the substitutions $B_i = T_i + xyy_i$ ($i=0,1,2$) in the system given in Lemma~\ref{lem:blocks_mis}.
The resulting system will be denoted by ${\cal S}$.
In theory one could also compute the determinant of the Jacobian of ${\cal S}$.
In this case however, its size prevents us to do so directly with a small computer.

Instead, we first consider the determinant of the Jacobian of the system in Lemma~\ref{lem:blocks_mis} without applying the substitutions in~\eqref{eq:Bi_sp_mis}.
So that each partial derivatives of the $B_i$'s are seen as unknown functions to be computed separately using~\eqref{eq:Bi_sp_mis}, then substituted in order to obtain ${\cal DJ(S)}$ as in cacti graphs.
And setting $y_0=y_1=1$ and $x=\rho_1$ in the system ${\cal S \cup DJ(S)}$, one gets the following approximation:
\begin{equation*}
    \rho_1 \approx 0.0770510356
\end{equation*}

To finally estimate the number of independent sets in a uniformly at random SP graph on $n$ vertices, we apply the same method as in the case of cacti graphs but now with the following estimate (see \cite[Section 3]{BGKN07}) of the number of SP graphs on $n$ vertices:
\begin{align*}
    & g_n \underset{x\to \rho}\sim g\cdot n^{-5/2}\cdot \rho^{-n}\cdot n!,
    & \text{where } \rho \approx 0.1102133467 \text{ and } g>0.
\end{align*}
This implies that the number of maximal independent sets in a random SP graph on $n$ vertices is asymptotically:
\begin{align*}
    & AI_n = \frac{|\mathcal{I}_n|}{g_n} \sim C \cdot \alpha^{n},
    & \text{where } C = A_1/g > 0 \text{ and } \alpha = \rho/\rho_1 \approx 1.4303941013.
\end{align*}

Now, after setting $x = \rho_1(y_0)$ in the system ${\cal S \cup DJ(S)}$, we can differentiate it to obtain a system defining $\rho_1'(y_0)$.
The only difference with the system for cacti graphs is that we now have to compute each derivative in ${\cal S \cup DJ(S)}$ by first computing it separately in \eqref{eq:Bi_sp_mis}.
From there, we can again set $y_0=y_1=1$ to substitute $\rho_1(1) = \rho_1$ with its approximated value.
This gives us a system characterising uniquely $\rho_1'(1)$ that we can again solve numerically.
In particular, we obtain:
\begin{equation*}
    \rho_1' = \rho_1'(1) \approx -0.0207425825.
\end{equation*}
So that the expected size of a maximal independent set in a uniformly at random SP graph on $n$ vertices is asymptotically:
\begin{align*}
    & \mathbb{E}[SI_n] = \mu n + O(1),
    & \text{where } \mu = -\frac{\rho_1}{\rho_1'} \approx 0.2690257588.
\end{align*}

\subsubsection{Matchings in 2--connected vertex-rooted series--parallel graphs}

The generating functions counting restricted RM--trees associated to the network decompositions of vertex--rooted SP graphs carrying a matching are derived similarly to the case of independent sets.
In this section, we will hence only briefly explicit the systems of equations corresponding to the different ways of rooting the restricted RM--trees.

\paragraph{\sl Restricted RM--tree rooted at an M--node.}

It holds that:
\begin{align*}
    & \bar T_i^M = x\left(y_0\bar M_{i0} + y_1\bar M_{i1} + y_2\bar M_{i2}\right)
    & (i\in\{0,1,2\}),
\end{align*}
where we set:
\begin{align*}
    & \bar M_{00} = \exp_{\geq 3}(\bar S_{00}), \\
    & \bar M_{01} = \bar S_{01}(y\exp_{\geq 1}(\bar S_{02}) + \exp_{\geq 2}(\bar S_{02})), \\
    & \bar M_{02} = y\exp_{\geq 2}(\bar S_{02})+\exp_{\geq 3} (\bar S_{02}), \\
    & \bar M_{11} = \exp_{\geq 2}(\bar S_{22})(y + \bar S_{11}) + \exp_{\geq 1}(\bar S_{22})(y\bar S_{11} + \bar S_{12}^2) + \exp(\bar S_{22})y\bar S_{12}^2, \\
    & \bar M_{12} = \bar S_{12}(y\exp_{\geq 1}(\bar S_{22}) + \exp_{\geq 2}(\bar S_{22})), \\
    & \bar M_{22} = y\exp_{\geq 2}(\bar S_{22})+\exp_{\geq 3}(\bar S_{22}).
\end{align*}

\paragraph{\sl Restricted RM--tree rooted at an R--node.}

For $i\in\{0,2\}$, it holds that:
\begin{equation*}
    \bar T_i^R = \frac{x^2}{2} ((\bar D_{i0} - \bar S_{i0})z_0\bar A_{0i} + (\bar D_{i1} - \bar S_{i1})z_1\bar A_{2i}  + (\bar D_{i2} - \bar S_{i2})(z_1\bar A_{1i} + z_2\bar A_{2i})),
\end{equation*}
where for $ij\in\{00,01,02,10,11,12,20,21,22\}$, we set:
\begin{equation*}
    \bar A_{ji} = (\bar D_{j0} - \bar S_{j0})z_0\bar D_{0i} + (\bar D_{j1} - \bar S_{j1})z_1\bar D_{2i} + (\bar D_{j2} - \bar S_{j2})(z_1\bar D_{1i} + z_2\bar D_{2i}).
\end{equation*}
Finally, for $i=1$ we have:
\begin{align*}
    \bar T_1^R = \,\,
    & \frac{x^2}{2} ( (\bar D_{10} - \bar S_{10})z_0\bar A_{02} + (\bar D_{20} - \bar S_{20})z_0\bar A_{01} + (\bar D_{11} - \bar S_{11})z_1\bar A_{22} \\
    & + (\bar D_{12} - \bar S_{12})(z_12\bar A_{12} + z_2\bar A_{22}) + (\bar D_{22} - \bar S_{22})(z_1\bar A_{11} + z_2\bar A_{21})).
\end{align*}

\paragraph{\sl Restricted RM--tree rooted at an edge.}

Similar restrictions, as for independent sets, between edges of the restricted RM--tree apply in the case of matchings:
\begin{align*}
        & \bar T_i^{R-M} = x(z_0\bar S_{i0}\bar P_{i0} + z_1(\bar S_{i1}\bar P_{i2} + \bar S_{i2}\bar P_{i1}) + z_2\bar S_{i2}\bar P_{i2})
        \qquad\qquad (i\in\{0,2\}), \\
        & \bar T_1^{R-M} = x(z_0(\bar S_{10}\bar P_{20} + \bar S_{20}\bar P_{10}) + z_2(\bar S_{12}\bar P_{22} + \bar S_{22}\bar P_{12}) + z_1(\bar S_{11}\bar P_{22} + 2\bar S_{12}\bar P_{21} + \bar S_{22}\bar P_{11})).
\end{align*}

\paragraph{\sl Vertex--rooted 2--connected series--parallel graphs carrying a matching.}

And we can finally apply the dissymmetry theorem for trees to obtain:
\begin{align}\label{eq:Bi_sp_match}
    & \bar B_i - \bar e_i = \bar T_i = \bar T_i^M + \bar T_i^R - \bar T_i^{R-M}
    & (i\in\{0,1,2\}),
\end{align}
where $\bar e_0 = xyy_2$, $\bar e_1 = xyy_1$ and $\bar e_2 = xy(y_0 + y_2)$.

\subsubsection{Number and expected size of maximal matchings}

Here again, the size of the system obtained by substituting the equations from \eqref{eq:Bi_sp_match} into the system of Lemma \ref{lem:blocks_mm} prevents us to directly compute the determinant of its Jacobian.
Instead, we also consider the determinant of the Jacobian of the system of Lemma \ref{lem:blocks_mm} without substitution and compute each partial derivatives directly on each equation of \eqref{eq:Bi_sp_match}, then apply the substitutions.

From the resulting system, one can then obtain the following approximation:
\begin{equation*}
    \rho_2\approx 0.0749665399.
\end{equation*}
Ths implies that the number of maximal matchings in a uniformly at random SP graph on $n$ vertices is asymptotically:
\begin{align*}
    & AM_n \sim D \cdot \beta^{n},
    & \text{where } D > 0 \text{ and } \beta \approx 1.4701671808.
\end{align*}

Concerning the expected size of a maximal matching, remember that we must now consider the perturbation with respect to the variable $z_1^2$.
And similarly to before, we get:
\begin{equation*}
    \rho_2' = \rho_2'(1) \approx -0.0478172197.
\end{equation*}
So that the expected size of a maximal matching in a random SP graph on $n$ vertices is asymptotically:
\begin{align*}
    & \mathbb{E}[SM_n] = \lambda n + O(1),
    & \text{where } \lambda = -\frac{1}{2}\cdot\frac{\rho_2}{\rho_2'} \approx 0.3189237476.
\end{align*}

\end{document}